\documentclass[a4paper]{amsart}

\usepackage[utf8]{inputenc}
\usepackage{amsmath, amsthm, amssymb, amsfonts}
\usepackage[T1]{fontenc}
\usepackage{tikz-cd} 
\usepackage{enumitem}
\usepackage[italic]{mathastext}
\usepackage[english]{babel} 
\usepackage{url}
\usepackage{amsmath}

\usepackage{tcolorbox}
\usepackage{xcolor}
\usepackage{hyperref}
\theoremstyle{plain}
\newtheorem{thm}{Theorem}[section]

\newtheorem{lemma}[thm]{Lemma}
\newtheorem{prop}[thm]{Proposition}
\newtheorem{corollary}[thm]{Corollary}
\theoremstyle{definition}
\newtheorem{rmk}[thm]{Remark}
\newtheorem{defi}[thm]{Definition}
\newtheorem{example}[thm]{Example}
\newtheorem{question}[thm]{Question}

\usepackage{upgreek}
\usepackage{mathtools}
\usepackage[all]{xy}
\usepackage{filecontents}
\DeclareMathOperator{\br}{Br}

\DeclareMathOperator{\ev}{ev}

\DeclareMathOperator{\spec}{Spec}
\DeclareMathOperator{\rsw}{rsw}

\DeclareMathOperator{\Pic}{Pic}
\DeclareMathOperator{\Ev}{Ev}
\DeclareMathOperator{\im}{\mathfrak{m}}
\DeclareMathOperator{\ip}{\mathfrak{p}}

\DeclareMathOperator{\h}{H}

\DeclareMathOperator{\fil}{fil}
\newcommand{\numberset}{\mathbb}
\newcommand{\p}{\numberset{P}}
\newcommand{\F}{\numberset{F}}

\newcommand{\Z}{\numberset{Z}}
\newcommand{\Q}{\numberset{Q}}

\newcommand{\Os}{\mathcal{O}}
\newcommand{\Ad}{\mathbf{A}}

\newcommand{\A}{\mathcal{A}}

\newcommand{\et}{\text{ét}}
\usepackage{dsfont}
\usepackage{commath}
\newcommand{\thistheoremname}{}
\newtheorem*{genericthm*}{\thistheoremname}
\newenvironment{namedthm*}[1]
  {\renewcommand{\thistheoremname}{#1}%
   \begin{genericthm*}}
  {\end{genericthm*}}

\makeatletter

\definecolor{mame}{rgb}{0.0, 0.5, 0.5}
\definecolor{bibi}{rgb}{0.79, 0.08, 0.48}

\title{The role of primes of good reduction in the Brauer--Manin obstruction}
\author{Margherita Pagano}
\date{July 2025}
\begin{document}
\maketitle
\markboth{\small Margherita Pagano}{\small The role of primes of good reduction in the Brauer--Manin obstruction}
\begin{abstract}
    \noindent We discuss the role of primes of good reduction in the existence of the Brauer--Manin obstruction to weak approximation for varieties defined over number fields. Following Bright and Newton, we give conditions on the ramification index of primes involved in the Brauer--Manin obstruction. To support the results, many examples of transcendental nature on K$3$ surfaces are given. 
\end{abstract}
\section{Introduction}
Let $k$ be a number field and $\Ad_k$ be the ring of adèles of $k$, i.e.\ the restricted product of $k_\nu$ for all places $\nu$ of $k$, taken with respect to the rings of integers $\Os_\nu\subseteq k_\nu$. Let $V$ be a smooth, proper, geometrically irreducible variety over $k$. In this paper, we are interested in the image of $V(k)$ in the set of the adèlic points $V(\Ad_k)$. We say that $V$ satisfies \emph{weak approximation} if the image of $V(k)$ in $V(\Ad_k)$ is dense. In $1970$ Manin \cite{Manin} introduced the use of the Brauer group to study the image of $V(k)$ in $V(\Ad_k)$, and in $1977$ Colliot-Thélène and Sansuc used the same idea to study weak approximation, see \cite[Proposition~5]{CTSansuc}. More precisely, they showed that there exists a pairing
\[
    \langle-,-\rangle\colon \br(V)\times V(\Ad_k)\rightarrow \Q/\Z
\]
such that the closure of the rational points of $V$ lies in the right kernel of the pairing, denoted by $V(\Ad_k)^{\br}$. If $V(\Ad_k)^{\br}$ is not equal to the whole $V(\Ad_k)$ we say that there is a \emph{Brauer--Manin obstruction to weak approximation} on $V$.

For every non-archimedean place $\ip$ and for every element $\A\in \br(V)$ we denote by $\ev_\A$ the map $\langle \A, -\rangle\mid_{V(k_{\ip})}\colon V(k_{\ip})\rightarrow \Q/\Z$.
\begin{defi}
    We say that a prime $\ip$ of $k$ \emph{plays a role} in the Brauer--Manin obstruction to weak approximation on $V$ if there exists an element $\A\in \br(V)$ such that the corresponding evaluation map $\ev_\A\colon V(k_{\ip})\rightarrow \br(k_{\ip})$ is non-constant.  
\end{defi}
In this paper, we are addressing the following question: 
\begin{question}
    Assume $\Pic(\bar{V})$ to be torsion-free and finitely generated. Which primes can play a role in the Brauer--Manin obstruction to weak approximation on $V$?
\end{question}
This question is inspired by a question asked by Swinnerton-Dyer \cite[Question 1]{colliotskogoodred}. He asked whether only primes of bad reduction and archimedean places can play a role in the Brauer--Manin obstruction to weak approximation. 
 Bright and Newton \cite{BrightNewton} prove the following theorem, which leads to a negative answer to Swinnerton-Dyer's question.
\begin{namedthm*}{Theorem C, \cite{BrightNewton}}\label{thm: Theorem C}
    Assume that $\h^0(V,\Omega^2_V)\ne 0$. Let $\ip$ be a prime at which $V$ has good ordinary reduction and of residue characteristic $p$. Then there exists a finite field extension $k'/k$, a prime $\ip'$ of $k'$ lying over $\ip$ and an element $\A\in \br(V_{k'})\{p\}$ such that the corresponding evaluation map is non-constant on $V(k'_{\ip'})$. 
\end{namedthm*}
Hence, under the assumptions of Theorem C, up to a base change to a finite field extension $k'/k$ we can always find a prime $\ip'$ of good reduction that plays a role in the Brauer--Manin obstruction to weak approximation on $V_{k'}$.

\vspace{2mm}

The main focus of this paper is on the field extension $k'/k$ appearing in Bright and Newton's result, with a particular emphasis on K$3$ surfaces. We analyse how the reduction type and the absolute ramification index are involved in the possibility for a prime of good reduction to play a role in the Brauer--Manin obstruction to weak approximation. 
\begin{thm}\label{intro_thm: ordinary good reduction}
    Let $\ip$ be a prime of good ordinary reduction for $V$ of residue characteristic $p$. Assume that the special fibre at $\ip$, $\mathcal{V}(\ip)$ has no non-trivial global $1$-forms, $\h^1(\overline{\mathcal{V}(\ip)},\Z/p\Z)=0$ and $(p-1)\nmid e_{\ip}$. Then the prime $\ip$ does not play a role in the Brauer--Manin obstruction to weak approximation on $V$.
\end{thm}
As pointed out in \cite[Remark~11.5]{BrightNewton}, if $V$ is a K$3$ surface, then for every prime of good reduction $\ip$ the special fibre is still a K$3$ surface and hence has no non-trivial global $1$-forms and satisfies $\h^1(\overline{\mathcal{V}(\ip)},\Z/p\Z)=0$. It follows from Theorem \ref{intro_thm: ordinary good reduction} that for K$3$ surfaces over number fields, in order for a prime of good ordinary reduction to be involved in the Brauer--Manin obstruction to weak approximation it is necessary that $(p-1)\mid e_{\ip}$. In the case of K$3$ surfaces, we analyse also what happens for primes of good non-ordinary reduction.
\begin{thm}\label{intro_thm: good non-ordinary reduction}
    Let $V$ be a K$3$ surface and $\ip$ be a prime of good non-ordinary reduction for $V$ with $e_{\ip} \leq (p-1)$. Then the prime $\ip$ does not play a role in the Brauer--Manin obstruction to weak approximation on $V$.
\end{thm}

Furthermore, in Section~\ref{subsubsection: On the existence of a Brauer Manin obstruction over a field extension} we improve the result of Theorem~C under some extra assumption on the crystalline cohomology of the special fibre at a prime $\ip$ of good ordinary reduction. In particular, we show that it is always possible to find (over a finite field extension $k'/k$) an element whose order is \emph{exactly} $p$ and whose evaluation map is non-constant, cf. Theorem \ref{thm: obstruction after field extension}.

We provide several examples in support of the conditions on the ramification index found in Theorem \ref{intro_thm: ordinary good reduction} and \ref{intro_thm: good non-ordinary reduction}. In particular, we exhibit K$3$ surfaces $V$ over number fields such that:
\begin{enumerate}[label=(\alph*)]
    \item $V$ has good ordinary reduction at a prime $\ip$ with ramification index $e_{\ip}=p-1$ and there is an element $\A\in \br(V)[p]$ whose evaluation map is non-constant on $V(k_{\ip})$, see Theorem~\ref{thm: family of examples} with $\alpha=1$;
    \item $V$ has good ordinary reduction at a prime $\ip$ with $e_{\ip}=p-1$ and $\ip$ does not play a role in the Brauer--Manin obstruction to weak approximation, see Example~\ref{example: Kummer}; 
    \item $V$ has good non-ordinary reduction at a prime $\ip$ with $e_{\ip}=p$ and there is an element $\A\in \br(V)[p]$ whose evaluation map is non-constant on $V(k_{\ip})$, see Theorem~\ref{thm: family of examples} with $\alpha=\sqrt{2}$.
\end{enumerate}
More precisely: from (a) we get that the condition $(p-1)\nmid e_{\ip}$ in Theorem~\ref{intro_thm: ordinary good reduction} is necessary; from (b) we get that the converse of Theorem~\ref{intro_thm: ordinary good reduction} does not hold in general; from (c) we get that the inequality in Theorem~\ref{intro_thm: good non-ordinary reduction} cannot be refined.

Finally, we point out that in all the examples of K$3$ surfaces in which a prime of good reduction plays a role in the Brauer--Manin obstruction to weak approximation, the corresponding element in the Brauer group is of \emph{transcendental} nature, i.e. it does not belong to the algebraic Brauer group, which is defined as the kernel of the natural map from $\br(V)$ to $\br(\bar{V})$, where $\bar{V}$ is the base change of $V$ to an algebraic closure of $k$ (cf. Corollary~\ref{cor: refined Swan conductor and algebraic elements}). 
The first example of a transcendental element in the Brauer group of a K$3$ surface defined over a number field was given by Wittenberg in \cite{Wittenberg}. Other examples of transcendental elements that obstruct weak approximation can be found in \cite{HasVarWAorder2}, \cite{IeronymouOrd2},\cite{PreuOrd3}, \cite{NewtonOrd3}, \cite{BergVarOrd3}, \cite{IeronymouSkoro} and \cite{ErrataIeronymouSkoro}. In all these articles, the obstruction to weak approximation comes from the fact that the transcendental algebra has non-constant evaluation at either the place at infinity or at a prime of bad reduction. The first example of a K$3$ surface in which a prime of good reduction plays a role in the Brauer--Manin obstruction to weak approximation can be found in \cite{Pagano}. 
\subsection{Notation} If $G$ is an abelian group and $n$ a positive integer, then $G[n]$ denotes the kernel of multiplication by $n$ on $G$. If $p$ is a prime, then $G\{p\}$ denotes the $p$-power torsion subgroup of $G$. For any smooth scheme $X$ over a field $k$, we denote by $\Z/n\Z(r)$ the element of the bounded derived category $D^b(X_{\et})$ defined by
\[
\Z/n\Z(r):=
\begin{cases}
    \mu_n^{\otimes r}, \text{ if } k \text{ has characteristic }0;\\
    \mu_m^{\otimes r}\oplus W_s \Omega^r_{X,\log}[-r], \text{ if }k\text{ has characteristic }p>0\text{ and }n=mp^s, \, p\nmid m.
\end{cases}
\]
For the definition of $W_s \Omega^r_{X,\log}$ see Illusie \cite{Illusie}. 

Finally, if $R$ is a $k$-algebra and either $n$ is invertible in $k$ or $R$ is smooth over $k$, we define
\begin{equation}\label{eq: definition of h^q(K)}
    \h^q_n(R):=\h^q(R_\et, \Z/n\Z(q-1)) \, \text{ and } \, \h^q(R):=\varinjlim_n \h^q_n(R). 
\end{equation}
We denote by $k$ any number field and by $\Os_k$ its ring of integers. Let $\Omega_k$ be the set of places of $k$, for any finite subset $S\subset \Omega_k$ we define $\Os_{k,S}$ as the intersection in $k$ of the local rings $\Os_{\nu}$ for $\nu \in S$. We denote by $V$ any proper, smooth and geometrically integral variety over a number field $k$. Let $\mathcal{V}$ be an $\Os_{k,S}$-scheme of finite type such that there exists an isomorphism between $\mathcal{V}\times_{\Os_{k,S}} \spec(k)$ and $V$ (i.e. $\mathcal{V}$ is an $\Os_{k,S}$ model for the scheme $V$); for any finite place $\ip$ not in $S$, we denote by $\mathcal{V}(\ip)$ the special fibre at $\ip$, i.e. the base change of $\mathcal{V}$ to $\spec(k(\ip))$, by $\mathcal{V}_{\ip}$ the base change of $\mathcal{V}$ to $\Os_{\ip}$ and by $V_{\ip}$ the base change of $V$ to $k_{\ip}$.

Let $p$ be a prime number. We denote by $L$ any $p$-adic field (i.e. a finite field extension of the field of $p$-adic numbers $\Q_p)$, with ring of integers $\Os_L$ and residue field $\ell$, which is a finite field of positive characteristic. We denote by $X$ any smooth and geometrically integral variety over $L$ and by $\mathcal{X}$ any smooth model over $\Os_L$ with geometrically integral fibre, i.e. $\mathcal{X}$ is such that there exists an isomorphism between $\mathcal{X}\times_{\Os_L} \spec(L)$ and $X$. Moreover, we denote by $Y$ the special fibre of $\mathcal{X}$, which is the base change of $\mathcal{X}$ to $\ell$ and is a smooth, proper and geometrically integral variety over the finite field $\ell$, since smoothness and properness are stable under base change.

\subsection{Outline}
Section~\ref{Section: the refined Swan conductor} is devoted to introducing the notion and some properties of the refined Swan conductor. In Section~\ref{section:ordinarygoodred} we prove Theorem~\ref{intro_thm: ordinary good reduction} and we prove a stronger version of \cite[Theorem~C]{BrightNewton} for K$3$ surfaces. Section~\ref{section: ordinary case examples} is devoted to the explanation of several examples of K$3$ surfaces in which primes of good ordinary reduction are (or are not) involved in the Brauer--Manin obstruction to weak approximation. In Section~\ref{Section: non-ordinary good reduction} we give a proof of Theorem~\ref{intro_thm: good non-ordinary reduction}. Finally,  in Section~\ref{Section: family of examples} we exhibit a family of examples where, depending on a parameter $\alpha$, both ordinary and non-ordinary reduction occur and we analyse what happens to the Brauer--Manin obstruction to weak approximation depending on the parameter $\alpha$.  

\subsection{Acknowledgements}
I would like to thank Martin Bright for all the useful conversations, suggestions and having read carefully the first version of this paper. I am thankful to Evis Ieronymou for pointing out to me Example~\ref{subsection: non ordinary case Example}. This work was written during my PhD at Leiden University and is part of my PhD thesis. I would like to thank Evis Ieronymou, Rachel Newton and Alexei Skorobogatov because their revision of my thesis greatly improved this article. Finally, I am grateful to the anonymous referee for all the extremely helpful comments.
\section{The refined Swan conductor}\label{Section: the refined Swan conductor}
Let $V$ be a smooth proper and geometrically integral variety over a number field $k$. In order to study the role of a prime $\mathfrak{p}$ in the Brauer--Manin obstruction, we need to understand, for a given element $\A\in \br(V)$, the behavior of the corresponding evaluation map 
\[
    \ev_\A\colon V(k_{\mathfrak{p}})\rightarrow \Q/\Z.
\]
If we denote by $\mathrm{res}\colon \br(V)\rightarrow\br(V_{\mathfrak{p}})$ the natural restriction map, then for every point $P\in V(k_{\mathfrak{p}})$ we have a corresponding point $P_{\mathfrak{p}}\in V_{\mathfrak{p}}(k_{\mathfrak{p}})$ and
\[
    \ev_{\mathrm{res}(\A)}(P_{\mathfrak{p}})=\ev_\A(P).
\]
Hence, if the evaluation map attached to $\A$ is non-constant on $V(k_{\ip})$, then also $\mathrm{res}(\A)\in \br(V_{\mathfrak{p}})$ has non-constant evaluation map on $V_{\mathfrak{p}}(k_{\mathfrak{p}})$.

\vspace{2mm}

\subsection{P-adic setting}\label{Section: General Setting} Let $p$ be a prime number and $L$ a finite field extension of $\Q_p$, with ring of integers $\Os_L$, uniformiser $\pi$ and residue field $\ell$. Let $X$ be a smooth and geometrically irreducible $L$-variety having good reduction (i.e. there exists a smooth proper $\Os_L$-scheme $\mathcal{X}$ whose generic fiber is isomorphic to $X$). We assume furthermore the special fiber $Y:=\mathcal{X}\times_{\spec(\Os_L)}\spec(\ell)$ to be geometrically irreducible, 
\begin{equation}\label{eq:diagram1}
    \begin{tikzcd}
    X\arrow[d]\arrow[r,"j"] &\mathcal{X}\arrow[d] &Y\arrow[d] \arrow[l,"i"'] \\
    \spec(L)\arrow[r] &\spec(\Os_L) &\spec(\ell).\arrow[l]
    \end{tikzcd}
\end{equation}

In \cite{BrightNewton} Bright and Newton define the following filtration, called the \emph{Evaluation filtration}, on the Brauer group of $X$:
 \begin{alignat*}{2}
     &\Ev_n \br X:=\{ \mathcal{B} \in \br(X) \mid \forall L'/L \text{ finite, } \forall P\in \mathcal{X}(\Os_{L'})\\
     &\qquad \qquad \qquad \qquad \quad \ev_{\mathcal{B}}  \text{ is constant on }B(P,e_{L'/L}(n+1))\}, \qquad (n\geq 0)\\ 
     &\Ev_{-1} \br X :=\{ \mathcal{B} \in \br(X) \mid \forall L'/L \text{ finite, } \ev_{\mathcal{B}} \text{ is constant on }\mathcal{X}(\Os_{L'})\}\\
     &\Ev_{-2} \br X :=\{ \mathcal{B} \in  \br(X) \mid \forall  L'/L  \text{ finite, } \ev_{\mathcal{B}}  \text{ is zero on }\mathcal{X}(\Os_{L'})\}
 \end{alignat*}
For every positive integer $m$ we denote by $\Ev_n \br(X)[m]$ the restriction of $\Ev_n \br(X)$ to $\br(X)[m]$, i.e. $\Ev_n \br(X)[m]:=\Ev_n \br(X) \cap \br(X)[m]$. 

\vspace{2mm}

If $(m,p)=1$ Colliot-Th\'{e}l\`{e}ne--Skorobogatov \cite{colliotskogoodred} and Bright \cite{BrightBadReduction} have proven that $\Ev_0\br(X)[m]=\br(X)[m]$. Moreover, there exists a map $\partial_m$, called  \emph{residue map}, that fits in the following sequence
\begin{equation*}
    0\rightarrow \br(\mathcal{X})[m]\rightarrow \br(X)[m]\xrightarrow{\partial_m} H^1(Y,\Z/m\Z), \quad \text{\cite[Chapter~3,(3.17)]{BGgroupTheleneSkoro}}.
\end{equation*}
The residue map is such that 
\begin{align*}
    &\Ev_{-1}\br(X)[m]=\{\A \in \br(X)[m]\mid \partial_m(\A)\in H^1(\ell,\Z/m\Z)\} \\
    &\Ev_{-2} \br(X)[m]=\{ \A \in \br(X)[m] \mid \partial_m (\A)=0\},
\end{align*} 
see \cite{CTSaito} for a proof of it.

In order to give a description of the Evaluation filtration also on the $p$-power torsion part of $\br(X)$, Bright and Newton introduce the filtration $\{\fil_n \br(X)\}_{n\geq 0}$ on $\br(X)$, defined through the notion (introduced by Kato \cite{Kato}) of Swan conductor on discrete Henselian valuation fields, see Section \ref{subsection: definition of the refined swan conductor}. The interaction between the two filtrations is described in the theorem that follows.

\begin{thm}\label{thm: evaluation fil and refined Swan conductor}
There exists a residue map 
\[
    \partial\colon \fil_0\br(X)\rightarrow \varinjlim_n H^1(Y,\Z/n\Z)=:H^1(Y,\Q/\Z)
\]
such that
    \begin{enumerate}[label=(\alph*)]
    \item $\fil_0\br(X)$ coincides with $\Ev_0\br(X)$; 
    \item $\Ev_{-1} \br(X)=\{\A \in \br(X)\mid \partial(\A)\in H^1(\ell,\Q/\Z)\}$;
    \item $\Ev_{-2} \br(X)=\{ \A \in \br(X) \mid \partial (\A)=0\}.$
    \end{enumerate} 
    For every $n \geq 1$ there is a map
    \[
    \rsw_{n,\pi}\colon\fil_n \br(X)\rightarrow \h^0(Y,\Omega^2_Y)\oplus \h^0(Y,\Omega^1_Y)
    \]
    called refined Swan conductor, such that
    \[
    \Ev_n \br(X)=\left\{
        \A\in \fil_{n+1}\br(X)\mid \rsw_{n+1,\pi}(\A)\in \h^0(Y,\Omega^2_Y)\oplus 0\right\}.
    \]
\end{thm}
\begin{proof}
    This is a reformulation of \cite[Theorem~A]{BrightNewton}.
\end{proof}
It is clear that to understand the evaluation filtration on the Brauer group we need to understand the residue map $\partial$ and the refined Swan conductor maps $\rsw_{n,\pi}$.

\subsection{Definition of the refined Swan conductor}\label{subsection: definition of the refined swan conductor}

We denote by $K$ a Henselian discrete valuation field of characteristic zero with ring of integers $\Os_K$ and residue field $F$ of characteristic $p$. Let $\pi$ be a uniformiser in $\Os_K$ and $\mathfrak{m}$ be the maximal ideal of $\Os_K$.
Let $A$ be a ring over $\Os_K$, $R:=A/\im A$ and $i,j$ the inclusions of the special and generic fibers into $\spec(A)$:
\begin{center}
    \begin{tikzcd}
     \spec(A\otimes_{\Os_K} K)\arrow[r,"j"] &\spec(A) &\spec(R).\arrow[l,"i"']
    \end{tikzcd}
\end{center}
We define
\[
    V_n^q(A):=\h^q\left((R)_\et, i^*\mathrm{R}j_* \Z/n\Z(q-1)\right)
\]
and $V^q(A):=\varinjlim_n V_n^q(A)$. 

\begin{rmk}
    The natural map in $D^b(A_\et)$
\[
\mathrm{R}j_* \Z/n\Z(q-1)\rightarrow i_* i^* \mathrm{R} j_* \Z/n\Z(q-1)
\]
induces a natural map $\h^q_n(A\otimes_{\Os_K} K)\rightarrow V_n^q(A)$ for all $n,q$. Gabber \cite{Gabber} proved that this map is an isomorphism if $(A,\mathfrak{m} A)$ is a Henselian pair. In this case, using the Kummer map $(A\otimes_{\Os_K} K)^\times \rightarrow \h^1(A\otimes_{\Os_K} K, \Z/n\Z(1))$ and the cup product, we define a product
  \begin{align*}
        V_n^q(A)\times &((A\otimes_{\Os_K} K)^\times)^{\oplus r} \rightarrow V_n^{q+r}(A)\\
        (\chi,&a_1,\dots,a_r) \mapsto \{\chi, a_1,\dots,a_r\}.
\end{align*}
For a general $A$, the isomorphism \cite[Remark~1.2.13]{PhDthesis}
\begin{equation}\label{eq: isomorphism with Henselianisation}
    V_n^q(A)\simeq V_n^q(A^{(h)})
\end{equation} 
allows to extend the product above. From now on we identify $V^q_n(A)$ and $V^q_n(A^{(h)})$. In particular, note that if $A=\Os_K[T]$, then all the polynomials of the form $1+\pi^n p(T)$ are invertible in $A^{(h)}$, see \cite[0EM7]{Stacks} for an overview on the henselisation of (not necessarily local) rings.
\end{rmk}
 
\begin{defi}\label{def: Swan Conductor}
The increasing filtration $\{\fil_n \h^q(K)\}_{n\geq 0}$ on $\h^q(K)$ is defined by 
$$\chi \in \fil_n \h^q(K) \Leftrightarrow \{\chi, 1+\pi^{n+1} T\}=0 \text{ in }V^{q+1}(\Os_K[T]). $$
We say that $\chi \in \h^q(K)$ has \emph{Swan conductor} $n$, if $\chi\in \fil_n\h^q(K)$ and $\chi \notin \fil_{n-1} \h^q(K)$. By \cite[Lemma 2.2]{Kato} $\h^q(K)=\cup_n \fil_n\h^q(K)$, hence the Swan conductor is defined for every element $\chi\in \h^q(K)$.
\end{defi}

In \cite{Kato} and \cite{BrightNewton} certain maps $\lambda_\pi\colon \h^{q}_n(R)\oplus \h^{q-1}_n (R)\rightarrow V^q_n(A)$ are defined under additional assumptions on the $\Os_K$-algebra $A$. We provide an overview of these maps, which, with a slight abuse of notation, we will always refer to as $\lambda_\pi$.
\begin{rmk}\label{rmk: different definition of lambda_pi}
    \leavevmode

    \begin{itemize}
        \item In \cite[Section~1.4]{Kato} Kato defines for every $n$ an injective map 
        \[
        \lambda_\pi\colon \h^q_n(F)\oplus\h^{q-1}_n(F) \rightarrow \h^q_n(K).
        \]
        This collection of maps induces an injective map
        \[
            \lambda_\pi\colon \h^q(F)\oplus \h^{q-1}(F)\rightarrow \h^q(K).
        \]
        \item In \cite[Section~1.9]{Kato} Kato extends the definition of $\lambda_\pi$ to any smooth $\Os_K$-algebra $A$. In particular, he defines a map 
        \[
            \lambda_\pi\colon \h^q_p(R)\oplus \h^{q-1}_p(R)\rightarrow V^q_p(A).
        \]
        
        \item Finally, in \cite[Section~2.2]{BrightNewton} Bright and Newton generalise the previous map by defining 
        \[
        \lambda_\pi\colon  \h^q_{p^r}(R)\oplus \h^{q-1}_{p^r}(R)\rightarrow V^q_{p^r}(A)
        \]
        for any $r\geq 1$.
    \end{itemize}
    It is proven in \cite[Proposition~6.1]{Kato} that the image of 
\[
\lambda_\pi \colon \h^q_n(F)\oplus \h^{q-1}_n(F)\rightarrow \h^q_n(K)
\]
coincides with $\fil_0\h^q_n(K)$.
    In \cite[Section~1.3]{Kato} Kato shows that for every $r\geq 1$ there is a surjection $\delta_r$ from $W_r \Omega^q_{R}$ to $\h^q_{p^r}(R)$. Following \cite{Kato} and \cite{BrightNewton} we sometimes use $\lambda_\pi$ also to denote the composition
    \begin{equation}\label{eq: delta and cup product}
        W_r \Omega^q_{R}\oplus W_r \Omega^{q-1}_{R}\xrightarrow{\delta_r} V^q_{p^r}(R)\oplus V^{q-1}_{p^r}(R)\xrightarrow{\lambda_\pi} V^q_{p^r}(A).
    \end{equation}
\end{rmk}
The following theorem allows Kato to define the refined Swan conductor of $\chi\in \fil_n\h^q(K)$.
\begin{thm}\label{Thm: existence of refined Swan conductor}
    Let $\chi\in \fil_n \h^q(K)$, with $n\geq 1$; then there exists a unique pair $(\alpha,\beta)$ in $\Omega^q_F\oplus \Omega^{q-1}_F$ such that 
    \begin{equation}
        \{\chi,1+\pi^n T\}=\lambda_\pi(T\alpha,T\beta) \quad \text{in }V^{q+1}_p(\Os_K[T]).
    \end{equation}
\end{thm}
\begin{proof}
    See \cite[Section ~5]{Kato}.
\end{proof}
The uniqueness of the pair $(\alpha,\beta)$ in the previous theorem implies, for every $n\geq 1$, the existence of a homomorphism 
\[
    \rsw_{n,\pi}\colon\fil_n \h^q(K)\rightarrow \Omega^q_F\oplus \Omega^{q-1}_F
\]
whose kernel is $\fil_{n-1}\h^q(K)$. The pair $(\alpha,\beta)$ is called the \emph{refined Swan conductor} of $\chi \in \fil_n \h^q(K)$.

We end this section with the definition of the residue map $\partial$ from $\fil_0\h^q(K)$ to $\h^{q-1}(F)$ (cf. \cite[Section~2.5]{BrightNewton}, \cite[Section~7.5]{Kato}).
\begin{defi}\label{defi: residue map on fields K,F}
    The \emph{residue map} 
    \[
    \partial\colon \fil_0 \h^q(K)\rightarrow \h^{q-1}(F)
    \]
    is defined as the projection on the second component of the inverse of the isomorphism $\lambda_\pi$ from $\h^q(F)\oplus \h^{q-1}(F)$ to $\fil_0 \h^q(K)$. 
\end{defi}

\subsection{The image of the refined Swan conductor}\label{subsection: properties of the image of the refined Swan conductor}
To further describe the refined Swan conductor maps, we need to recall some properties of differential forms over fields of positive characteristic.

Let $\ell$ be a perfect field of characteristic $p>0$ and $R$ be an $\ell$-algebra. Let
\[
    Z^q_R:=\ker(d\colon\, \Omega^q_R\rightarrow \, \Omega^{q+1}_R) \quad \text{ and } \quad B^q_R:=\mathrm{im}(d\colon\, \Omega^{q-1}_R\rightarrow \, \Omega^{q}_R).
\]
\begin{lemma}[Inverse Cartier operator]\label{lemma:inverse cartier operator}
    Assume $R$ to be regular; then there exists a unique homomorphism of groups
    \[
        C^{-1}_R\colon \Omega^1_R \rightarrow \, \Omega^1_R/B^1_R
    \]
    satisfying
    \begin{itemize}
        \item $C_R^{-1}(da)=a^{p-1}da \mod B^1_R$ for all $a\in R$; 
        \item $C^{-1}_R(\lambda \omega)=\lambda^p C^{-1}_R(\omega)$ for all $\lambda\in R$;
        \item $d\circ C^{-1}_R=0$.
    \end{itemize} 
    Moreover, $C^{-1}$ induces an isomorphism from $\Omega^1_R$ to $Z^1_R/B^1_R$.
\end{lemma}
\begin{proof}
    See \cite[Theorem 1.3.4]{BrionKumar}.
\end{proof}
\begin{rmk}\label{remark: on the F^p structure on B^q and Z^q}
    The subgroup $B^1_R$ has a natural structure of an $R$-module, which is given by $\alpha \cdot d\beta=\alpha^p d\beta=d(\alpha^p \beta)$. If we denote by $^p\Omega^1_R$ the $R$-module structure on $\Omega^1_R$ given by $\alpha \cdot \omega=\alpha^p \omega$, then the condition $C_R^{-1}(\lambda \omega)=\lambda^p C_F^{-1}(\omega)$ is equivalent to asking that $C_R^{-1}\colon\Omega^1_R\rightarrow \,^p\Omega^1_R/^pB^1_R$ is a morphism of $R$-modules. 
\end{rmk}
We can extend the definition of $C^{-1}_R$ to higher differential forms by setting 
$$C^{-1}_R(\omega_1\wedge \dots \wedge \omega_q):=C^{-1}_R(\omega_1)\wedge\dots \wedge C^{-1}_R(\omega_q).$$
\begin{thm}\label{thm: higher dimension Cartier isomorphism}
    Let $R$ be regular; then the morphism
    $$C^{-1}_R\colon\Omega^q_R\rightarrow \, Z^q_R/B^q_R$$ 
    is an isomorphism for all $q\geq 0$. We denote by $C_R$ its inverse, which is called the {Cartier operator}.
\end{thm}
\begin{proof}
    See \cite[Theorem 1.3.4]{BrionKumar}.
\end{proof}

The following corollary gives a way to characterise exact differential forms in terms of the Cartier operator.
\begin{corollary}\label{cor: Exact Forms in terms of Cartier Op}
  Let $R$ be regular. A $q$-form $\omega\in \Omega^q_R$ is exact if and only if $d(\omega)=0$ and $C_R(\omega)=0$.  
\end{corollary}

The last object we need to define is the subgroup of {logarithmic} $q$-differential forms on $R$, which will play a crucial role in this paper. 
\begin{defi}\label{def: logarithmic forms}
    The \emph{logarithmic} $q$-differential forms on $R$, denoted by $\Omega^q_{R,\log}$, are defined as the kernel of the map
    $$C^{-1}_R-\mathrm{id}\colon \Omega^q_R\rightarrow \Omega^q_R/B^q_R.$$
    
\end{defi}
We have the following explicit description of logarithmic forms when $R$ is a field. 
\begin{thm}\label{thm: Log forms on field F}
    Let $F$ be a field, finitely generated over a perfect field $\ell$. The logarithmic differential $q$-forms $\Omega^q_{F,\log}$ is the subgroup of $\Omega^q_F$ generated by elements of the form 
    $$\frac{dy_1}{y_1}\wedge \dots \wedge\frac{dy_q}{y_q}, \text{ with }y_i\in F^\times.$$
\end{thm}
\begin{proof}
    It follows from the surjectivity in the Bloch--Gabber--Kato Theorem \cite[Theorem $9.5.2$]{GilleSzamuely}.
\end{proof}
If $R$ is smooth over a field of positive characteristic, we can identify $\h^q_p(R)=\h^1(R_\et, \Omega^{q-1}_{R,\log})$ with the cokernel of 
\[
{C}_{R}^{-1}-1\colon \Omega^{q-1}_R\rightarrow \Omega^{q-1}_R/B^{q-2}_R,
\]
see \cite[Section~1.3]{Kato}.
We denote by $\delta_1$ both the map from $\Omega^{q-1}_R/B^{q-2}_R$ to $\h^q_p(R)$ and its composition with the natural map $\h^q_p(R)\rightarrow \h^q(R)$. 

\vspace{2mm}

Let $e:=\mathrm{ord}_K(p)$ the absolute ramification index of $K$ and $e':=ep(p-1)^{-1}$. 
\begin{lemma}\label{lemma: rsw d(alpha)=0 and d(beta)=(alpha)}
    Let $\chi$ be an element in $\fil_n \h^q(K)$ with 
    \[
    \rsw_{n,\pi}(\chi)=(\alpha,\beta)\in \Omega^2_F\oplus \Omega^1_F.
    \]
    Then $d\alpha=0$ and $d\beta=(-1)^qn \alpha$.
\end{lemma}
\begin{proof}
    See \cite[Lemma~2.17]{BrightNewton}.
\end{proof}
\begin{rmk}\label{rmk: rsw comp proj_2}
    We get that:
    \begin{enumerate}[label=(\arabic*), ref= \thermk(\arabic*)]
        \item \label{rmk: rsw comp proj_2 1}if $p\mid n$, then $d\alpha=0$ and $d\beta=0$, meaning that $(\alpha,\beta)\in Z^q_F\oplus Z^{q-1}_F$;
        \item \label{rmk: rsw comp proj_2 2}if $p\nmid n$, then $\alpha=\bar{n}^{-1}d\beta$, with $\bar{n}$ residue of $n$ modulo $\pi$. In particular, the composition
        \[
        \fil_n \h^q(K)\xrightarrow{\rsw_{n,\pi}} \Omega^q_F\oplus \Omega^{q-1}_F \xrightarrow{\mathrm{pr}_2} \Omega^{q-1}_F
        \] 
        has also kernel equal to $\fil_{n-1}\h^q(K)$.
    \end{enumerate}
\end{rmk}
In \cite[Lemma~2.19]{BrightNewton} Bright and Newton link the refined Swan conductor of $\chi \in \fil_n\h^q(K)$ to the one of $p \cdot \chi$, whenever $n\geq e'$. More precisely, assume that $\rsw_{n,\pi}(\chi)=(\alpha,\beta)$ and let $\bar{u}$ be the reduction modulo $\pi$ of $p\cdot \pi^{-e}$, then $p\cdot \chi \in \fil_{n-e} \h^q(K)$ and
\begin{equation}\label{eq: rsw n with n>e'}
    \rsw_{n-e,\pi}(p\cdot \chi)=\begin{cases}
        (\bar{u}\alpha,\bar{u}\beta) \text{ if }n>e';\\
        (\bar{u}\alpha+C(\alpha),\bar{u}\beta+C(\beta)) \text{ if }n=e'.
    \end{cases}
\end{equation}
We prove a result analogous to the one proven by Bright and Newton for elements $\chi\in \fil_{np}\h^q(K)$, when $np<e'$.
\begin{lemma}\label{lemma: refined np with np<e'}
    Let $\chi\in \fil_{np} \h^q(K)$, with $np<e'$. Then $p\cdot \chi \in \fil_n\br(K)$ and if $\rsw_{np,\pi}(\A)=(\alpha,\beta)$, then $d\alpha=0$, $d\beta=0$ and 
    \[
        \rsw_{n,\pi}(p\cdot \chi)=(C(\alpha),C(\beta)).
    \]
\end{lemma}

\begin{proof}
    From Remark~\ref{rmk: rsw comp proj_2 1} we know that $(\alpha,\beta)\in Z^q_F\oplus Z^{q-1}_F$, since clearly $p\mid np$. The condition $e'>np$ implies 
    \begin{equation}\label{eq: e-np+n>0}
        e-np+n>0.
    \end{equation}
    Let $u\in \Os_K^\times$ be such that $p=u\cdot \pi^e$. We have
    
    \[
        \{ p\cdot \chi, 1+\pi^{n+1} T \}=\{ \chi, (1+\pi^{n+1} T)^p \}=\{\chi, 1+\pi^{np+1}b(T)\}
    \]
    where
    \[
        b(T)=\frac{(1+\pi^{n+1} T)^p-1}{\pi^{np+1}}.
    \]
    We can rewrite $b(T)$ as 
    \[
        \sum_{k=1}^{p-1} \pi^{e+(n+1)k-(np+1)} a_k T^k+\pi^{p(n+1)-(np+1)} T^p, \quad \text{for some $a_k\in \Z$. }
    \]
    Note that, for every $1\leq k \leq p$, we have that from the inequality~\eqref{eq: e-np+n>0} 
    \[
        e+(n+1)k-np-1=(e-np)+(n+1)k-1\geq 0.
    \]
    Therefore, $b(T)\in \Os_K [T]$. Now, since by assumption $\chi \in \fil_{np} \h^q(K)$, we have that $\{\chi, 1+\pi^{np+1}b(T)\}=0$ for all $b(T)\in \Os_{K}[T]$, thus $p\cdot \chi \in \fil_n\h^q(K)$. In a similar way,
    \[
        \{ p\cdot \chi, 1+\pi^n T \}=\{ \A, (1+\pi^n T)^p \}=\{\chi, 1+\pi^{np}c(T)\}
    \]
    where 
    \[
        c(T)=\frac{(1+\pi^n T)^p-1}{\pi^{np}}.
    \]
    We can rewrite $c(T)$ as 
    \[
       \sum_{k=1}^{p-1} \pi^{e+nk-np} a_k T^k+ T^p, \quad \text{for some $a_k\in \Z$.}
    \]
    In this case, again the inequality~\eqref{eq: e-np+n>0} together with $1\leq k\leq p-1$, implies $e+nk-np>0$. Therefore, $c(T)\in \Os_K[T]$ and its reduction modulo $\pi$ is equal to $T^p$. It follows from \cite[(6.3.1)]{Kato} that 
    \[
        \{\chi, 1+\pi^{np} c(T)\}=\lambda_\pi(\bar{c}(T)\alpha,\bar{c}(T) \beta)=\lambda_\pi(T^p\alpha,T^p \beta)=\lambda_\pi(TC(\alpha),TC(\beta))
    \]
    where the last equality follows from \cite[Lemma~2.18(2)]{BrightNewton}.
\end{proof}
\begin{rmk}\label{rmk: Chapter 1, image of p torsion element via refined Swan conductor}
    For every non-negative integer $d$, we denote by $\fil_n \h^q_d(K)$ the intersection of $\fil_n \h^q(K)$ with $\h^q_d(K)$. In \cite{Kato} Kato proves that for every non-negative integer $d$ prime to $p$,
$\h^q_d(K)=\fil_0\h^q_d(K)$, see \cite[Proposition~6.1]{Kato}. We now prove some useful properties of the refined Swan conductor on $p$-power order elements.\footnote{Parts (2) and (3) are already mentioned by Kato in \cite[Sections 4 and 5]{Kato}.} 
\begin{enumerate}[label=(\arabic*), ref= \thermk(\arabic*)]
    \item \label{rmk: Chapter 1, image of p torsion element via refined Swan conductor 1} The filtration $\fil_n \h^q_{p^m}(K)$ is finite. 

    \vspace{1mm}
    
    Assume first $e'$ to be an integer, i.e. $(p-1)\mid e$. For $m=1$ it is proven in \cite[Proposition~4.1]{Kato} that $\h^q_p(K)=\fil_{e'}\h^q_p(K)$. This is equivalent to saying that for all $n>e'$ and $\chi\in \fil_n\h^q_p(K)$, we have $\rsw_{n,\pi}(\chi)=(0,0)$. Assume that $m>1$, $\chi\in \fil_n\h^q_{p^m}(K)$ with $n> e'+(m-1)e$ and $\rsw_{n,\pi}(\chi)=(\alpha,\beta)$. From equation~\eqref{eq: rsw n with n>e'} we know that $\rsw_{n-e,\pi}(p\cdot \chi)=(\bar{u}\alpha,\bar{u}\beta)$, hence working by induction on $m$ we get that $\rsw_{n,\pi}(\chi)=(0,0)$. This result is essentially already proven in \cite[Proposition~17]{ieronymou2021evaluation}. 
    
    If $e'$ is not an integer, then we take the field extension $K(\zeta)/K$, with $\zeta$ a primitive $p$-root of unity $\zeta$. We denote by $e_K$ and $e_{K(\zeta)}$ the absolute ramification indices of $K$ and $K(\zeta)$, respectively, and by $e'_K$ and $e'_{K(\zeta)}$ the products $e_K \cdot p \cdot (p-1)^{-1}$ and $e_{K(\zeta)} \cdot p \cdot (p-1)^{-1}$, respectively. Let $n> e'_K +(m-1)\cdot e_K$ and $\chi \in \fil_n \h^q_{p^m}(K)$, then 
    \[
        e_{K(\zeta)/K} \cdot n > e'_{K(\zeta)} +(m-1)e_{K(\zeta)}.
    \]
    Since $(p-1)\mid e'_{K(\zeta)}$, we get (from what we said above) that
    \[
    \rsw_{e_{K(\zeta)/K} \cdot n,\pi}(\mathrm{res}(\chi))=(0,0)
    \]
    where $\mathrm{res}$ is the natural map from $\h^q_{p^m}(K)\rightarrow \h^q_{p^m}(K(\zeta))$.
    It follows from \cite[Lemma~2.16]{BrightNewton} that 
    \[
        \rsw_{e_{K(\zeta)/K} \cdot n,\pi}(\mathrm{res}(\chi))=(\bar{a}^{-n}(\alpha+\beta \wedge d\log \bar{a}), \bar{a}^{-n} e_{K(\zeta)/K} \beta)
    \]
    with $\bar{a}$ invertible in the residue field of $K(\zeta)$ (cf. Section~\ref{subsection: Refined Swan conductor and extension of the base field}).
    Hence, since $p\nmid e_{K(\zeta)/K}$ we can conclude that $\rsw_{n,\pi}(\chi)=(0,0)$. 

    \vspace{1mm}
    
    \item \label{rmk: Chapter 1, image of p torsion element via refined Swan conductor 2} Let $\chi \in \fil_{np}\h^q_p(K)$ with $np<e'$ and $\rsw_{np,\pi}(\chi)=(\alpha,\beta)$. We know from Lemma~\ref{lemma: refined np with np<e'} that $(\alpha,\beta)\in Z^q_F\oplus Z^{q-1}_F$; since $\chi$ has order $p$
    \[
        (C(\alpha),C(\beta))=\rsw_{n,\pi}(p\cdot \chi)=(0,0)
    \]
    Equivalently, from Corollary \ref{cor: Exact Forms in terms of Cartier Op} we get that for $np<e'$, the refined Swan conductor on the $p$-torsion takes image in $B^q_F\oplus B^{q-1}_F$.  

    \vspace{1mm}
    
    \item \label{rmk: Chapter 1, image of p torsion element via refined Swan conductor 3} Assume $e'$ to be an integer, $\chi\in \fil_{e'}\h^q_p(K)$ and $\rsw_{e',\pi}(\chi)=(\alpha,\beta)$. From equation~\eqref{eq: rsw n with n>e'} we get that
    \[
        -\bar{u}\alpha=C(\alpha) \; \text{ and } \; -\bar{u}\beta=C(\beta).
    \]
    Let $\zeta$ be a primitive $p$-root of unity and $c=(\zeta-1)^p\pi^{-e'}$, then $c=(c_1)^p$, with $c_1=(\zeta-1)\pi^{-e/(p-1)}$. By the properties of the Cartier operator we have
    \[
        C(\bar{c} \alpha)=\bar{c}_1C(\alpha)=-\bar{u}\bar{c}_1 \alpha.
    \]
    Since 
    $(\zeta-1)^{p-1}\equiv -p \mod \pi^{e+1}$ \footnote{In fact, $1=((\zeta-1)+1)^p=(\zeta-1)^p +p (\zeta-1)^{p-1}+\binom{p}{2}(\zeta-1)^{p-2}+\dots +p\cdot (\zeta-1)+1$. Since $\zeta\ne 1$ we get $(\zeta-1)^{p-1}+p(\zeta-1)^{p-2}+\binom{p}{2}(\zeta-1)^{p-3}+\dots +p=0$. Now the congruence follows from the fact that $\mathrm{val}((\zeta-1)\cdot p)\geq 1+e$.},
    $(\zeta-1)^{p-1}=-p+\pi^{e+1} v$ for some $v\in \Os_K$. Hence
    \[
        c=\frac{\zeta-1}{\pi^{e/(p-1)}}\cdot \frac{(\zeta-1)^{p-1}}{\pi^e}=c_1\cdot\left(\frac{-p+\pi^{e+1} v}{\pi^e}\right)=c_1\cdot (-u+\pi v).
    \]
    Thus $\bar{c}=-\bar{u}\bar{c}_1$ and
    \[
        \mathrm{mult}_{\bar{c}}\circ \rsw_{e',\pi}(\h^q_p(K))\subseteq \Omega^q_{F,\log} \oplus \Omega^{q-1}_{F,\log}
    \]
    where $\mathrm{mult}_{\bar{c}}$ is the map from $\Omega^q_F$ to $\Omega^q_F$ sending a $q$-form $\omega$ to $\bar{c}\cdot \omega$. 
\end{enumerate}
    
\end{rmk}
\subsection{Swan conductor on Br(X)}\label{subsection: refined Swan conductor on br(X)}

 Let $K^h$ be the field of fractions of the henselisation of the discrete valuation ring $\Os_{\mathcal{X},Y}$. 
 Note that 
\[
\h^2(K^h)=\varinjlim_n \h^2(K^h_\et,\Z/n\Z(1))=\br(K^h).
\]
Bright and Newton \cite{BrightNewton} define the filtration $\{\fil_n\br(X)\}_{n\geq 0}$ on $\br(X)$ as the pull-back via the natural map $\br(X)\rightarrow \br(K^h)$ of the filtration $\{\fil_n\br(K^h)\}_{n\geq0}$ on $\br(K^h)$. Therefore, it is possible to extend the definition of the residue map and the refined Swan conductor to elements in $\br(X)$ simply as the residue map and refined Swan conductor of the image of $\A$ in $\br(K^h)$. A priori these maps take values in $\h^1(F)$ and $\Omega^2_F\oplus \Omega^1_F$ respectively. However, Bright and Newton prove that the residue map of an element in $\fil_0\br(X)$ takes values in $\h^1(Y,\Q/\Z)\subseteq \h^1(F)$ (see \cite[Proposition~3.1(1)]{BrightNewton}) and that the refined Swan conductor on $\fil_n\br(X)$ takes image in $\h^0(Y,\Omega^2_Y)\oplus \h^0(Y,\Omega^1_Y)\subseteq \Omega^2_F\oplus \Omega^1_F$ (a proof following \cite[Theorem~7.1]{Kato} can be found in \cite[Theorem~B]{BrightNewton}). 

The aim of this subsection is to transfer the results we got in Section~\ref{subsection: properties of the image of the refined Swan conductor} on the refined Swan conductor on $\br(K^h)$ to the refined Swan conductor on $\br(X)$.

\vspace{2mm}

It is possible to generalise the definition of inverse Cartier operator on a smooth and proper variety $Y$ defined over a perfect field $\ell$ of positive characteristic. Following Illusie \cite{Illusie}, we denote by $F_Y$ the absolute Frobenius endomorphism of $Y$ and by $Y^{(p)}$ the base change of $Y$ via the absolute Frobenius $\sigma_\ell$ of the base field $\ell$, namely
\begin{center}
    \begin{tikzcd}
        Y \arrow[dr] \arrow[r,"F_{Y/\ell}"] &Y^{(p)}\arrow[d] \arrow[r,"W"] &Y\arrow[d]\\
        &\spec(\ell) \arrow[r,"\sigma_\ell"] &\spec(\ell)
    \end{tikzcd}
\end{center}
where $W\circ F_{Y/\ell}=F_Y$; we call $F_{Y/\ell}$ the relative Frobenius of $Y$ over $\ell$ and denote by $\Omega^{\bullet}_{Y/\ell}$ the De Rham complex of $Y/\ell$.
\begin{rmk}
    In \cite{Illusie}, Illusie works more generally with $S$-schemes, where the base $S$ is a scheme of positive characteristic. Here, we are only interested in varieties over perfect fields and in this case the absolute De Rham complex $\Omega^\bullet_Y$ coincides with the relative De Rham complex $\Omega^\bullet_{Y/\ell}$ (since $\ell$ perfect implies $\Omega^q_{\ell}=0$). 
\end{rmk}

For every $q\geq 0$ we define:
\[
Z^q_{Y}:=\ker(d\colon \Omega^q_{Y}\rightarrow \Omega^{q+1}_{Y}) \quad \text{and} \quad B^q_{Y}:=\mathrm{im}(d\colon\Omega^{q-1}_{Y}\rightarrow \Omega^{q}_{Y}).
\]
The differentials $d\colon\Omega^q_{Y}\rightarrow \Omega^{q+1}_{Y}$ are $\Os_{Y^{(p)}}$-linear, hence the sheaves
\[
    (F_{Y/\ell})_* Z^q_Y \quad \text{and} \quad (F_{Y/\ell})_* B^q_Y 
\]
are $\Os_{Y^{(p)}}$-modules and the sheaf of abelian groups $ \mathcal{H}^q((F_{Y/\ell})_*\Omega^{\bullet}_{Y})$ is also a sheaf of $\Os_{Y^{(p)}}$-modules.
\begin{defi}
    For every $q\geq 0$ there is a morphism of $\Os_Y$-modules, called the \emph{inverse Cartier operator}
$$C^{-1}_Y\colon\Omega^q_{Y}\rightarrow W_*    \mathcal{H}^q((F_{Y/\ell})_* \Omega^{\bullet}_{Y}).$$

\end{defi}
\begin{rmk}
    Since $\Omega^q_{Y^{(p)}}=W^* \Omega^q_{Y}$, by adjunction we get a morphism of $\Os_{Y^{(p)}}$-modules
    $$C^{-1}_{Y/\ell}\colon \Omega^q_{Y^{(p)}}\rightarrow \mathcal{H}^q((F_{Y/\ell})_*\Omega^{\bullet}_{Y}).$$
\end{rmk}

\begin{thm}
If $Y$ is a smooth variety over $\ell$, then $C_{Y/\ell}^{-1}$ is an isomorphism of $\Os_{Y^{(p)}}$-modules.
\end{thm}
\begin{proof}
    See \cite[Theorem $0.2.1.9$]{Illusie}.
\end{proof}
From now on we assume $Y$ to be smooth and proper over $\ell$. In this case, we denote by $C_{Y/\ell}$ the inverse of $C^{-1}_{Y/\ell}$. 
\begin{defi}
    For every non-negative integer $q$, we denote 
    \[
        \Omega^q_{Y,\log}:=\ker(W^*-C_{Y/\ell} \colon Z^q_{Y}\rightarrow \Omega^q_{Y^{(p)}}).
    \]
    The sheaf $\Omega^q_{Y,\log}$ is called the sheaf of \emph{logarithmic} $q$-forms on $Y$.
\end{defi}

\begin{thm}\label{thm: log forms on variety Y}
The sheaf $\Omega^q_{Y,\log}$ is the subsheaf of $\Omega^q_{Y}$ generated $\et$ale-locally by the logarithmic differentials, i.e. the sections of the form 
$$\frac{dy_1}{y_1}\wedge \dots \wedge \frac{dy_q}{y_q} \text{ with }y_i\in \Os_Y^*.$$ 
\end{thm}
\begin{proof}
    See \cite[Theorem $0.2.4.2$]{Illusie}.
\end{proof} 
The Cartier operator induces an exact sequence
\[
0\rightarrow \h^0(Y,B^q_Y) \rightarrow \h^0(Y,Z^q_Y)\xrightarrow{C_{Y}} \h^0(Y,\Omega^q_Y)
\]
We have the following lemma.
\begin{lemma}\label{lemma: (a) global form is log (b) global form is exact}
    Let $\omega \in \h^0(Y,\Omega^q_Y)$, then:
    \begin{enumerate}[label=(\arabic*), ref=\thelemma(\arabic*)]
        \item \label{lemma: (a) global form is log (b) global form is exact 1} $\omega\in \h^0(Y,\Omega^q_{Y,\log})$ if and only if the image of $\omega$ in $\Omega^q_F$ lies in $\Omega^q_{F,\log}$;
        \item \label{lemma: (a) global form is log (b) global form is exact 2} $\omega\in \h^0(Y,B^q_Y)$ if and only if the image of $\omega$ in $\Omega^q_F$ lies in $B^q_F$.
    \end{enumerate}
\end{lemma}
\begin{proof}
    The group $\Omega^q_{F,\log}$ can also be computed as the kernel of 
    \[
        \mathrm{id}-C\colon Z^q_F\rightarrow \Omega^q_F
    \]
    see \cite[Lemma~1.5]{MilneDuality}. Moreover, the natural inclusion of $\h^0(Y,\Omega^q_Y)$ in $\Omega^q_F$ is compatible with the differential maps and with the Cartier operator. The proof of $(1)$ follows from the definition of logarithmic forms, while $(2)$ is an immediate consequence of Corollary~\ref{cor: Exact Forms in terms of Cartier Op}. 
\end{proof}

\begin{corollary}\label{cor: on the image of the refined swan conductor}
    Let $u=p \pi^{-e}\in \Os_L^\times$; then one of the following cases occurs:
    \begin{enumerate}[label=(\arabic*), ref=\thecorollary(\arabic*)]
        \item \label{cor: on the image of the refined swan conductor 1} If $p\nmid n$, then 
        \[
            \mathrm{pr}_2\circ \rsw_{n,\pi}\colon \fil_n \br(X)\rightarrow \h^0(Y,\Omega^1_Y) 
        \]
        has kernel equal to $\fil_{n-1}\br(X)$.
         \item \label{cor: on the image of the refined swan conductor 2} If $p\mid n$ and $n<e'$ we write $n=mp$, then $(\alpha,\beta)$ lies in $ \h^0(Y,Z^2_Y)\oplus \h^0(Y,Z^1_Y)$ and the following diagram
        \begin{equation*}
            \begin{tikzcd}
                \fil_n \br(X) \arrow[d,"(-)^{\otimes p}"]\arrow[r,"\rsw_{n,\pi}"] &\h^0(Y,Z^2_Y)\oplus \h^0(Y,Z^1_Y)\arrow[d,"C"]\\
                \fil_m \br(X) \arrow[r,"\rsw_{m,\pi}"] &\h^0(Y,\Omega^2_Y)\oplus \h^0(Y,\Omega^1_Y).
            \end{tikzcd}
        \end{equation*}
        commutes.
        \item \label{cor: on the image of the refined swan conductor 3} If $n=e'$, then $(\alpha,\beta)$ lies in $ \h^0(Y,Z^2_Y)\oplus \h^0(Y,Z^1_Y)$ and the following diagram
        \begin{equation*}
            \begin{tikzcd}
                \fil_{e'} \br(X) \arrow[d,"(-)^{\otimes p}"]\arrow[r,"\rsw_{n,\pi}"] &\h^0(Y,Z^2_Y)\oplus \h^0(Y,Z^1_Y)\arrow[d,"\mathrm{mult}_{\bar{u}}+C"]\\
                \fil_{e'-e} \br(X) \arrow[r,"\rsw_{n-e,\pi}"] &\h^0(Y,\Omega^2_Y)\oplus \h^0(Y,\Omega^1_Y).
            \end{tikzcd}
        \end{equation*}
        commutes. Moreover, let $\bar{c}\in \ell^\times$ be the reduction of $c=(\zeta-1)^p\pi^{-e'}\in \Os_L^\times$, then $\bar{c}$ is such that 
        \begin{equation*}
                \mathrm{mult}_{\bar{c}}\left(\rsw_{e',\pi} \fil_{e'}\br(X)[p]\right)\subseteq \h^0(Y,\Omega^2_{Y,\log})\oplus \h^0(Y,\Omega^1_{Y,\log}).
        \end{equation*}
        \item \label{cor: on the image of the refined swan conductor 4} If $p\mid n$ and $n>e'$, then $(\alpha,\beta)$ lies in $\h^0(Y,Z^2_Y)\oplus \h^0(Y,Z^1_Y)$ and the following diagram
        \begin{equation*}
            \begin{tikzcd}
                \fil_n \br(X) \arrow[d,"(-)^{\otimes p}"]\arrow[r,"\rsw_{n,\pi}"] &\h^0(Y,Z^2_Y)\oplus \h^0(Y,Z^1_Y)\arrow[d,"\mathrm{mult}_{\bar{u}}"]\\
                \fil_{n-e} \br(X) \arrow[r,"\rsw_{n-e,\pi}"] &\h^0(Y,\Omega^2_Y)\oplus \h^0(Y,\Omega^1_Y).
            \end{tikzcd}
        \end{equation*}
        commutes.
    \end{enumerate}
\end{corollary}
\begin{proof}
    All these properties are consequences of Section~\ref{subsection: properties of the image of the refined Swan conductor}. More precisely: $(1)$ follows from Remark \ref{rmk: rsw comp proj_2 2}; $(2)$ follows from Lemma \ref{lemma: refined np with np<e'}; $(3)$ and $(4)$ are a direct consequence of \cite[Lemma~$2.19$]{BrightNewton}.
\end{proof}
\begin{rmk}
    For $\mathcal{F}=\Omega^q_Y,Z^q_Y$ or $B^q_F$, the cohomology group $\h^q(Y,\mathcal{F})$ in the Zariski or small $\et$ale site are isomorphic, and hence we will not specify over which site we are computing them. 
\end{rmk}
\subsection{Refined Swan conductor on cyclic algebras}
This section is needed to understand the computation in the examples appearing in this paper, it is not used to prove the main theorems.

Let $K$ be a Henselian field of characteristic $0$ with ring of integers $\Os_K$, uniformiser $\pi$ and residue field $F$ of positive characteristic $p$. Here we assume the field $K$ to contain a primitive $p$-root of unity $\zeta$. It follows from the Merkurjev-Suslin Theorem \cite[Theorem 8.6.5]{GilleSzamuely} that in this case $\br(K)[p]$ is generated by the classes of \emph{cyclic algebras}. The study of the $p$-primary part of the Galois
cohomology of a complete discrete valuation field K in terms of its residue field F,
which is not necessarily perfect goes back to Kato \cite{Kato2} (see \cite{ColliotTheleneSurvey} for a survey on the topic).

\begin{defi}
   Let $a,b\in K^\times$, then the $K$-algebra $(x,y)_p$ defined as 
\[
    (x,y)_p:=\langle a,b\mid a^p=x, \, b^p=y, \, ab=\zeta ba\rangle,
\]
is a central simple algebra, see \cite[Section~2.5]{GilleSzamuely}. With an abuse of notation we denote by $(x,y)_p$ also the corresponding equivalence class in $\br(K)[p]$.  
\end{defi}
It is possible to realise $(x,y)_p$ also as the cup product of an element $\chi_x\in \h^1_p(K)$ with $\delta(y)$, where $\delta$ is the boundary map $K^\times \rightarrow \h^1(K,\Z/p\Z(1))$ coming from the Kummer sequence, a proof can be found in the proof of Proposition~4.7.1 \cite{GilleSzamuely}.

\subsubsection{Bloch-Kato filtration}

The map from $\Z/p\Z$ to $\Z/p\Z(1)$ sending $1$ to $\zeta$ induces an isomorphism\begin{equation}\label{eqIsomoprhismRootUnity}
    \br(K)[p]\simeq \h^2\left(K,\Z/p\Z(2)\right)=:h^2(K).
\end{equation}
For any two non-zero elements $x,y\in K$ we denote by $\{x,y\}\in h^2(K)$ the cup product of $\delta(x)$ with $\delta(y)$. 

\vspace{2mm}

Bloch and Kato \cite{BlochKatoEtale} define a decreasing filtration $\{U^m h^2(K)\}_{m\geq 0}$ on $h^2(K)$ as follows: $U^0h^2(K)=h^2(K)$ and for $m\geq 1, \, U^m h^2(K)$ is the subgroup of $h^2(K)$ generated by symbols of the form 
\[
    \{1+\pi^m x, y\}, \text{ with }x\in \Os_K \text{ and }y\in K^\times.
\]
They describe the graded factors of the filtration
\[
    \mathrm{gr}^m:=\frac{U^mh^2(K^h)}{U^{m+1}h^2(K^h)}
\]
in terms of differential forms on the residue field $F$. 
\begin{prop}\label{propGradedPieces}
We have the following description of the graded factors $\mathrm{gr}^m.$
\begin{enumerate}[label=(\arabic*), ref=\theprop.(\arabic*)]
    \item\label{propGradedPieces 1} $U^m h^2(K^h)=\{0\}$ for $m> e'$; $U^{e'}h^2(K)$ coincides with the image of the injective map
    \begin{align*}
        \lambda_\pi \colon \h^2_p(F)\oplus \h^{1}_p(F)&\rightarrow \mathrm{gr}^{e'}=h^2(K)\\
        \delta_1\left[ \bar{x} \cdot d \log \bar{y} \right] &\mapsto \{1+(\zeta-1)^px,y\}\\
        \delta_1 \left[ \bar{x}\right] &\mapsto \{1+(\zeta-1)^px,\pi\}
    \end{align*}
    where $x$ and $y$ are any lifts of $\bar{x}$ and $\bar{y}$ to $K$ and $\delta_1[\omega]$ as in equation~\eqref{eq: delta and cup product}.  
    \item \label{propGradedPieces 2} Let $0<m<e'$ and $p\nmid m$. Then we have an isomorphism 
    \begin{align*}
        \rho_m\colon  \Omega^{1}_F &\xrightarrow{\simeq} \mathrm{gr}^m\\
        \bar{x} \cdot d \log \bar{y} &\mapsto \{1+\pi^m x,y\}
    \end{align*}
    where $x$ and $y$ are any lifts of $\bar{x}$ and $\bar{y}$ to $K$.  
    \item \label{propGradedPieces 3} Let $0<m<e'$ with $p\mid m$. Then we have an isomorphism 
    \begin{align*}
        \rho_m\colon \Omega^{1}_F/Z^1_F\oplus \Omega^0_F/Z^0_F &\xrightarrow{\simeq} \mathrm{gr}^m\\
        (\left[\bar{x} \cdot 
        d \log \bar{y}\right],0) &\mapsto \{1+\pi^m x,y\}\\
        (0,\left[\bar{x}\right])&\mapsto \{1+\pi^m x,\pi\}
    \end{align*}
    where $x$ and $y$ are any lifts of $\bar{x}$ and $\bar{y}$ to $K$, and $[\omega]$ is the image of $\omega\in \Omega^i_F$ in $\Omega^i_F/Z^i_F$, with $i=0,1$.  
    \item\label{propGradedPieces 4} We have an isomorphism 
    \begin{align*}
        \rho_0\colon \Omega^{2}_{F,\log}\oplus \Omega^{1}_{F,\log} &\xrightarrow{\simeq} \mathrm{gr}^0\\
        (d\log \bar{y}_1\wedge d \log \bar{y}_2,0) &\mapsto \{y_1,y_2\}\\
        (0,d \log \bar{y})&\mapsto \{ y,\pi\}
    \end{align*}
    where $y,y_1$ and $y_2$ are any lifts of $\bar{y},\bar{y}_1$ and $\bar{y}_2$ to $K$.  
\end{enumerate}
\end{prop}
\begin{proof}
     See \cite[Section~5]{BlochKatoEtale}. More precisely, in \cite[Lemma~5.1]{BlochKatoEtale} Bloch and Kato prove \ref{propGradedPieces 1}. However, instead of $\lambda_\pi$ (cf. Remark~\ref{rmk: different definition of lambda_pi}) they have the morphism $\rho_{e'}$, and they just prove afterwards \cite[equation~(5.15.1)]{BlochKatoEtale} that $\rho_{e'}$ induces the map $\lambda_\pi$. Finally, in \cite[Lemma~5.2 and 5.3]{BlochKatoEtale} the rest of the proposition is proven.
\end{proof}
Note that since by assumption $K$ contains a primitive $p$-root of unity $\zeta$, the ramification index $e=\mathrm{ord}_K(p)$ is divisible by $(p-1)$ and hence $e'=ep(p-1)^{-1}$ is divisible by $p$. 
The key result of this section is the following proposition, which allows to use Bloch and Kato filtration to compute the refined Swan conductor. 
\begin{prop}\label{prop: rsw and map rho_m}
For every $0\leq m \leq e'$ we have that the isomorphism~\eqref{eqIsomoprhismRootUnity} induces an isomorphism
\[
\fil_{e'-m}\br(K)[p]\simeq U^m h^2(K).
\]
    Let $\bar{c}$ be the reduction modulo $\pi$ of $c:=\pi^{-e'}(\zeta-1)^{p}$. For $m<e'$ the compositions $\rsw_{e'-m}\circ \rho_m$ are as follows:
    \begin{align*}
        &\rsw_{e',\pi}(\rho_0(\alpha,\beta))=(\bar{c}\alpha,\bar{c}\beta) \\
        &\rsw_{e'-m,\pi}(\rho_m(\beta))=(\bar{c}d \beta,(e'-m)\bar{c}\beta), \, &\text{if } p\nmid m \\
        & \rsw_{e'-m,\pi}(\rho_m(\alpha,\beta))=(\bar{c}d \alpha,\bar{c}d\beta), \, &\text{if } p\mid m 
    \end{align*}
    where $\rho_m$ is the map defined in the previous proposition.
\end{prop}
Proposition~\ref{prop: rsw and map rho_m} is a rephrasing of a result proven in Kato~\cite[Lemma~4.3]{Kato}\footnote{A (quite technical) proof of it with some extra details can be found in \cite[Section 4.2.1]{PhDthesis}.}.

\subsection{Refined Swan conductor and extension of the base field}\label{subsection: Refined Swan conductor and extension of the base field}
In this section we analyse the behaviour of the refined Swan conductor under field extension of the base field. In \cite{BrightNewton} Bright and Newton prove the following result.
\begin{lemma}\label{lemma: rsw of field extension K'/K}
    Let $K'/K$ be a finite extension of Henselian discrete valuation fields of ramification index $e$. Let $\pi'$ be a uniformiser in $K'$, $F'$ be the residue field of $K'$ and define $\bar{a}\in F'$ to be the reduction of $\pi(\pi')^{-e}$. Let $\chi \in \fil_n \br(K)$, and let 
    \[
    \mathrm{res}\colon \br(K)\rightarrow \br(K')
    \]
    be the restriction map. Then $\mathrm{res}(\chi)\in \fil_{en}\br(K')$ and if $\rsw_{n,\pi}(\chi)=(\alpha,\beta)$, then 
    \[
        \rsw_{en,\pi'}(\mathrm{res}(\chi))=(\bar{a}^{-n}(\alpha+\beta\wedge d\log(\bar{a}), \bar{a}^{-n}e \beta).
    \]
\end{lemma}
\begin{proof}
    See \cite[Lemma~2.7]{BrightNewton}.
\end{proof}
We are back to the general setting of Section~\ref{Section: General Setting}. The aim of this section is to prove the following Lemma. 
\begin{lemma}\label{lemma: base change L'/L}
Let $L'/L$ be a finite field extension, with ramification index $e_{L'/L}$. Let $\pi'$ be a uniformiser in $L'$ and $\ell'$ its residue field.  Let $\A\in \fil_n \br(X)$ and let 
\[
    \mathrm{res}\colon \br(X)\rightarrow \br(X_{L'})
\]
be the restriction map. Then $\mathrm{res}(\A)\in \fil_{e_{L'/L}n}\br(X_{L'})$ and if $\rsw_{n,\pi}(\A)=(\alpha,\beta)$ with $(\alpha,\beta)\in \h^0(Y,\Omega^2_Y)\oplus \h^0(Y,\Omega^1_Y)$, then 
\[
   \rsw_{e_{L'/L}n,\pi'}(\mathrm{res}(\A))=(\bar{a}^{-n}\alpha, \bar{a}^{-n}e_{L'/L}\beta)\in \h^0(Y_{\ell'},\Omega^2_{Y_{\ell'}})\oplus\h^0(Y_{\ell'},\Omega^1_{Y_{\ell'}}) 
\]
with $\bar{a}\in \ell'$ reduction of $\pi(\pi')^{-e_{L'/L}}$.
\end{lemma}
By definition of the refined Swan conductor on $\br(X)$  we have the following commutative diagram
\[
    \begin{tikzcd}
        \fil_n\br(X) \arrow[d] \arrow[r,"\rsw_{n,\pi}"] &\h^0(Y,\Omega^2_Y)\oplus \h^0(Y,\Omega^1_Y) \arrow[d,hookrightarrow]\\
        \fil_n\br(K^h)\arrow[r,"\rsw_{n,\pi}"] &\Omega^2_F\oplus \Omega^1_F.
    \end{tikzcd}
\]

We recall the construction of $K^h$: let $\eta$ be the generic point of $Y\subseteq\mathcal{X}$, then we define $R$ as the henselisation of the discrete valuation ring $\Os_{\mathcal{X},\eta}$ and $K^h$ as the fraction field of $R$. 
The construction of $\Os_{\mathcal{X},\eta}$ (and hence of $K^h$) is local on $\mathcal{X}$. Therefore we can assume $\mathcal{X}=\spec(A)$, with $A$ smooth $\Os_L$-algebra, $Y=\spec(A/\pi A)$; hence $\eta=(\pi)\in \spec(A)$ and $\Os_{\mathcal{X},\eta}=A_{(\pi)}$. We can re-write diagram~\eqref{eq:diagram1} as:
\begin{equation}\label{eq:diagramAffine}
    \begin{tikzcd}
        A\otimes_{\Os_L} L & \arrow[l] A \arrow[r] &A/\pi A\\
        L \arrow[u] & \arrow[l] \Os_L \arrow[r]\arrow[u] & \ell. \arrow[u]
    \end{tikzcd}
\end{equation}
We denote by $X'$, $\mathcal{X}'$ and $Y'$ the base change of $X$, $\mathcal{X}$ and $Y$ to $\spec(L'), \spec(\Os_{L'})$ and $\spec(\ell')$, respectively.  Let $(K')^h$ be the fraction field of $R'$, where $R'$ is the Henselianisation of the discrete valuation ring $\Os_{\mathcal{X}',Y'}$. In this setting
\[
    \frac{A}{\pi A}\otimes_{\ell} \ell' = \left(A\otimes_{\Os_{L}} \ell\right) \otimes_\ell \ell' =\left(A\otimes_{\Os_L} \Os_{L'} \right)\otimes_{\Os_{L'}} \ell' =\frac{A\otimes_{\Os_L} \Os_{L'}}{(1\otimes \pi')}.
\]
Thus, the generic point $\eta'$ of $Y'$ is the ideal generated by $1\otimes \pi'$ and $\Os_{\mathcal{X}',Y'}$ becomes the ring $(A\otimes_{\Os_L} \Os_{L'})_{(1\otimes \pi')}$. Using \cite[Lemma 04GH]{Stacks} and \cite[Lemma 05WP]{Stacks} we get that  $(K')^h/K^h$ is finite with ramification index $e_{L'/L}$

\begin{proof}[Proof of Lemma~\ref{lemma: base change L'/L}]
    It follows immediately from the observation above together with Lemma \ref{lemma: rsw of field extension K'/K} and the fact that since $\bar{a}\in \ell'$, which is a finite field, $d\log(\bar{a})=0$.
\end{proof}
Lemma~\ref{lemma: base change L'/L} gives a way to detect whether an element $\A\in \br(X)$ is transcendental.
\begin{corollary}\label{cor: refined Swan conductor and algebraic elements}
    Assume that $\A\in \fil_n \br(X)$ for some $n\geq 1$ is such that $\rsw_{n,\pi}(\A)=(\alpha,\beta)$ with $\alpha \ne 0$, then $\A\notin \br_1(X)$, i.e. $\A$ is a transcendental element in the Brauer group of $X$.
\end{corollary}
\begin{proof}
    Assume $\A$ to be in $\br_1(X)$; then by definition of $\br_1(X)$ there is a finite field extension $L'/L$ such that $\mathrm{res}(\A)=0$ in $\br(X_{L'})$, where $\mathrm{res}$ is the restriction map from $\br(X)$ to $\br(X_{L'})$. Let $e_{L'/L}$ be the ramification index of the extension, $\pi'$ be a uniformiser of $L'$ and $\ell'$ its residue field. We know from Lemma~\ref{lemma: base change L'/L} that 
        \[
        \rsw_{e(L'/L)n,\pi_{L'}}(\mathrm{res}(\A))=(\bar{a}^{-n}\cdot \alpha,\bar{a}^{-n}e_{L'/L}\cdot \beta)
        \]
        where $\bar{a}^{-n}\in (\ell')^\times$. Hence, $\rsw_{e(L'/L)n,\pi_{L'}}(\mathrm{res}(\A))\ne (0,0)$ and, therefore, $\mathrm{res}(\A)$ cannot be the trivial element.
\end{proof}
We end this section with a lemma that shows how the evaluation map behaves under base change without the assumption of good reduction for $X$. 
\begin{lemma}\label{lemma: base change no good-reduction}
    Let $X$ be a variety over a p-adic field $L$, not necessarily having good reduction, and let $\A\in \br(X)\{p\}$ be such that $\ev_\A\colon X(L)\rightarrow \br(L)$ is non-constant. Then for every field extension $L'/L$ with order co-prime to $p$ we have that $\mathrm{res}(\A)\in \br(X_{L'})$ has also a non-constant evaluation map $\ev_{\mathrm{res}(\A)}\colon X_{L'}(L')\rightarrow \br(L')$.
\end{lemma}
\begin{proof}
   Let $P,Q\in X(L)$ be such that 
    \[
        \ev_{\A}(P)\ne \ev_{\A}(Q).
    \]
    Denote by $P'$ and $Q'$ the base change of $P$ and $Q$ to $L'$.
    Then since $L'/L$ has degree co-prime to $p$ the restriction map $\br(L)\rightarrow \br(L')$ is injective on the $p$-primary torsion subgroup; hence
    \[
    \ev_{\mathrm{res}(A)}(P')\ne \ev_{\mathrm{res}(A)}(Q').
    \]
\end{proof}
\section{Ordinary good reduction }\label{section:ordinarygoodred}
Let $V$ be a smooth, proper and geometrically integral $k$-variety and $\mathcal{V}$ an $\Os_{k}$-model of $V$. The main result of this section is the following theorem.
\begin{thm}\label{thm: ordinary good reduction}
    Let $\ip$ be a prime of good ordinary reduction for $V$ of residue characteristic $p$. Assume that the special fibre $\mathcal{V}(\ip)$ has no non-trivial global $1$-forms, $\h^1(\overline{\mathcal{V}(\ip)},\Z/p\Z)=0$ and $(p-1)\nmid e_{\ip}$. Then the prime $\ip$ does not play a role in the Brauer--Manin obstruction to weak approximation on $V$.
\end{thm}
Let $X$ be the base change to $k_{\ip}$ of $V$, $\mathcal{X}$ the base change of $\mathcal{V}$ to $\Os_{\ip}$ and $Y$ the special fibre $\mathcal{V}(\ip)$. As already explained in Section~\ref{Section: General Setting} if we are able to show that for all elements $\A\in \br(X)$ the evaluation map $\ev_\A\colon X(k_{\ip})\rightarrow \Q/\Z$ is constant, then the prime $\ip$ does not play a role in the Brauer--Manin obstruction to weak approximation on $V$. 

\vspace{2mm}

We start with the definition of ordinary variety. 
\begin{defi}\label{def: ordinary var}
    Let $Y$ be a smooth, proper and geometrically integral variety over a perfect field $\ell$ of positive characteristic. We say that $Y$ is \emph{ordinary} if $\h^i(Y,B^q_Y)=0$ for every $i,q$.
\end{defi}

    As pointed out in \cite[Example~7.4]{BlochKatoEtale} if $Y$ is an abelian variety, then we get back the usual definition of ordinary, namely $Y$ is ordinary if and only if there is an isomorphism between $Y(\bar{\ell})[p]$ and $(\Z/p\Z)^{\mathrm{dim}(Y)}$. In general, the defintion can be extended in many different equivalent ways, see for example \cite[Proposition~7.3]{BlochKatoEtale}. In particular, this notion is related to the slopes of the Newton polygon of the $F$-crystal $\h^q_{\mathrm{cris}}(Y/W)$. Roughly speaking, the role played by the Dieudonné module for abelian varieties is replaced by the $F$-crystals $\h^q_{\mathrm{cris}}(Y/W)$, with $q\geq 0$. However, since for the aim of this section we can avoid introducing crystalline cohomology and the de Rham-Witt complex we prefer not to add too many technicalities and use the definition of ordinary variety requiring the least number of new concepts. 

\subsection{Proof of Theorem~\ref{thm: ordinary good reduction}}\label{subsection: proof of good ordinary reduction theorem}
Throughout this section, we assume that $X$ is such that its special fibre $Y$ is smooth, \emph{ordinary} and such that both $\h^0(Y,\Omega^1_Y)$ and $\h^1(\bar{Y},\Z/p\Z)$ are trivial. In this case, the Cartier operator gives a bijection between the global closed $q$-forms and the global $q$-forms on $Y$. In fact, for every $q$ the short exact sequence $0\rightarrow B^q_Y\rightarrow Z^q_Y\xrightarrow{C}\Omega^q_Y\rightarrow 0$ induces a long exact sequence in cohomology
\[
0 \rightarrow \h^0(Y,B^q_Y)\rightarrow \h^0(Y,Z^q_Y)\xrightarrow{C}\h^0(Y,\Omega^2_Y)\rightarrow \h^1(Y,B^q_Y)\rightarrow \dots .
\]
The condition that $Y$ is ordinary assures the vanishing of $\h^0(Y,B^q_Y)$ and $\h^1(Y,B^q_Y)$ and hence the bijectivity of the Cartier operator $C\colon\h^0(Y,Z_Y^q)\rightarrow \h^0(Y,\Omega^q_Y)$.
\begin{lemma}\label{LemmaOrdinaryK3Fil_n<e'}
    Let $n$ be an integer such that $0<n<e'$, with $e'=ep(p-1)^{-1}$. Then for every $\A\in \fil_n \br(X)$ we have
    \[
        \rsw_{n,\pi}(\A)=0.
    \]
\end{lemma}
\begin{proof}
    The assumption $\h^0(Y,\Omega^1_Y)=0$ together with Lemma \ref{cor: on the image of the refined swan conductor 1} assure us that if $p\nmid n$, then $\rsw_{n,\pi}(\A)=0$. Moreover, as already pointed out in Section~\ref{subsection: properties of the image of the refined Swan conductor} if $\A\in \br(X)$ has order coprime to $p$, then $\A\in \fil_0\br(X)$. Therefore we are reduced to the case $\A\in \fil_{n} \br(X)[p^r]$, with $r\geq 1$ and $p\mid n$. We prove the lemma by induction on $r$. If $r=1$ and $\rsw_{n,\pi}(\A)=(\alpha,\beta)$, by Corollary \ref{cor: on the image of the refined swan conductor 2} 
    \[
        (0,0)=\rsw_{n/p,\pi}(\A^{\otimes p})=(C(\alpha),C(\beta)).
    \]
    By Lemma~\ref{lemma: (a) global form is log (b) global form is exact 2} it follows that $(\alpha,\beta)\in \h^0(Y,B^2_Y)\oplus \h^0(Y,B^1_Y)$, which are zero since $Y$ is ordinary. Assume the result to be true for $r-1$ and let $\A\in \br(X)[p^r]$, then again by Corollary \ref{cor: on the image of the refined swan conductor 2} together with the induction hypothesis we have that 
    \[
        (0,0)=\rsw_{n/p,\pi}(\A^{\otimes p})=(C(\alpha),C(\beta))
    \]
    and once again we get that $(\alpha,\beta)\in \h^0(Y,B^2_Y)\oplus \h^0(Y,B^1_Y)$ and therefore $\alpha$ and $\beta$ have to be zero.
\end{proof}

\begin{lemma}
    Assume that $(p-1)\nmid e$; then for every $n\geq 1$ and for every $\A\in \fil_n \br(X)$ we have
    \[
        \rsw_{n,\pi}(\A)=0.
    \]
\end{lemma}
\begin{proof}
    Let $\A\in \fil_n \br(X)[p^r]$. By the previous lemma, we already know that the result is true if $n<e'$ and we know that the result is true if $p\nmid n$, see \ref{cor: on the image of the refined swan conductor 1}. Again, we work by induction on $r$. For $r=1$, we know from Remark~\ref{rmk: Chapter 1, image of p torsion element via refined Swan conductor 1} that if $n> e'$ then $\rsw_{n,\pi}(\A)=(0,0)$. Hence, we can assume $n\leq e'$. Since by assumption $e'$ is not an integer $n\leq e'$ is equivalent to $n<e'$; we can therefore conclude by applying the previous lemma. Let $\A\in \fil_n \br(X)[p^r]$ and $\rsw_{n,\pi}(\A)=(\alpha,\beta)$. By assumption we know that if $n\geq e'$ then $n>e'$. By Corollary \ref{cor: on the image of the refined swan conductor 3} 
    \[
        \rsw_{n-e}(\A^{\otimes p})=(\bar{u}\alpha,\bar{u}\beta)
    \]
    with $\bar{u}\in \ell^\times$ and the result follows from the induction hypothesis. 
\end{proof}
Summarising, we proved the following theorem.
\begin{thm}\label{thm: ordinary over p adic field }
        Assume that $Y$ is an ordinary variety with no global $1$-forms. If $(p-1)\nmid e$, then $\br(X)=\fil_{0}\br(X).$ 
\end{thm}
The assumption $\h^1(\bar{Y},\Z/p \Z)=0$ implies that $\Ev_{-1}\br(X)\{p\}=\fil_0\br(X)\{p\}$, see \cite[Lemma~11.3]{BrightNewton}. Hence, combining this result with Theorem~\ref{thm: ordinary over p adic field } we get the proof of Theorem~\ref{thm: ordinary good reduction}.
\subsection{Existence of Brauer--Manin obstruction over field extensions}\label{subsubsection: On the existence of a Brauer Manin obstruction over a field extension}

In this section, we prove a stronger version of \cite[Theorem~C]{BrightNewton}, under the additional assumption that the crystalline cohomology groups on the special fiber are torsion-free. We show that with this extra assumption it is always possible (after a finite field extension) to find an element of order \emph{exactly} $p$ with non-constant evaluation map.  

\begin{thm}\label{thm: obstruction after field extension}
    Let $V$ be a smooth proper and geometrically integral variety over a number field $k$ and $\ip$ be a prime of good ordinary reduction for $V$. Assume that 
    \begin{itemize}
        \item the special fibre $\mathcal{V}(\ip)$ at $\ip$ has $\h^0(\mathcal{V}(\ip),\Omega^2_{\mathcal{V}(\ip)})\ne \{0\}$;
        \item the crystalline cohomology groups $\h^i(\mathcal{V}(\ip)/W)$ are torsion free for every $i\geq 0$.
    \end{itemize}
    Then there exists a finite field extension $k'/k$ and an element $\A\in \br(V_{k'})[p]$ that obstructs weak approximation on $V_{k'}$.
\end{thm}
For the proof we first change setting and work over the $p$-adic field $L:=k_{\ip}$. As usual, we denote by $X,\mathcal{X}$ the base change of $V,\mathcal{V}$ to $k_{\ip}$ and $\Os_{\ip}$ respectively. Moreover, we denote by $Y$ the base change of $\mathcal{V}$ to $\ell:=k(\ip)$. 

\vspace{2mm}

Without loss of generality we can assume that $L$ contains a primitive $p$-root of unity. We fix, for every $q$, an isomorphism on $X_\et$ between $\Z/p\Z(1)$ and $\Z/p\Z(q)$. Define 

\[
    {\mathcal{M}}_1^q:={i}^*\text{R}^q{j}_*\left(\Z/p\Z(1)\right)
\]
where ${i}$ and ${j}$ denote the inclusion of the generic and special fibre (${X}$ and ${Y}$, respectively) in ${\mathcal{X}}$, cf. Section \ref{subsection: refined Swan conductor on br(X)}. 
\begin{rmk}
    The main reference for this section is \cite{BlochKatoEtale}. In \cite{BlochKatoEtale} Bloch and Kato link a “twisted” version of the sheaves ${\mathcal{M}}^q_1$ with the sheaf of logarithmic forms on ${Y}$. They work with the sheaves ${M}^q_1:={i}^* R^q {j}_*\left(\Z/p\Z(q)\right)$. However, since for every $q$ we fixed an isomorphism between $\Z/p\Z(1)$ and $\Z/p\Z(q)$, the results apply also to the sheaves ${\mathcal{M}}^q_1$.
\end{rmk}
The sheaves ${\mathcal{M}}^q_1$ fit in the spectral sequence of vanishing cycles \cite[0.2]{BlochKatoEtale}:
\[
    E^{s,t}_2:=\h^s({Y},{\mathcal{M}}^t_1)\Rightarrow \h^{s+t} \left({X},\Z/p\Z(1)\right).
\]
By comparing the spectral sequence of vanishing cycles for $X$ and $K^h$ we get the commutative diagram:
\[
\begin{tikzcd}
    \h^2\left(X, \Z/p\Z(1)\right) \arrow[r,"f"] \arrow[d] & \h^0(Y,\mathcal{M}_1^2) \arrow[d,"g"]\\
    \br(K^h)[p] \arrow[r,"\text{res}"] &\h^0(K^h,\br(K^h_{nr})[p]).
\end{tikzcd}
\]
Here we are using Hilbert's theorem 90 \cite[Theorem~1.3.2]{BGgroupTheleneSkoro} to identify $\br(K^h)[p]=\h^2(K^h,\Z/p\Z(1))$. In \cite[Lemma~3.4]{BrightNewton} it is proven that the map $g$ is injective. The map $f$ is defined as the composition of the projection $\h^2(X,\Z/p\Z(1))\twoheadrightarrow E^{0,2}_\infty$ with the inclusion map $E^{0,2}_\infty \hookrightarrow E^{0,2}_2$. 

The injectivity of $g$ gives a map $f_1\colon \br(X)[p]\rightarrow \h^0(Y,\mathcal{M}_1^2)$ such that the diagram 
\[
\begin{tikzcd}
    \h^2(X,\Z/p\Z(1)) \arrow[d]\arrow[r,"f"] &\h^0(Y,\mathcal{M}_1^2) \\
    \br(X)[p] \arrow[ur,"f_1"']
 \end{tikzcd}
\]
commutes, where the vertical arrow comes from the Kummer exact sequence \cite[Section~1.3.4]{BGgroupTheleneSkoro}. In fact, every element in $\h^2\left(X,\Z/p\Z(1)\right)$ coming from an element of $\Pic(X)$ is sent to $0$ by $f$, since its image in $\br(K^h)[p]$ comes from an element in $\Pic(K^h)$, and the map $\Pic(K^h)\rightarrow \br(K^h)[p]$ is trivial.

By \cite[Proposition~6.1]{Kato} we also know that $\ker({\mathrm{res}})=\fil_0\br(K^h)[p]$, hence we have the commutative diagram
\begin{equation}\label{diagram: vanishing cycles}
\begin{tikzcd}
    \fil_0\br(X)[p]\arrow[r]\arrow[d] &\br(X)[p] \arrow[r,"f_1"] \arrow[d] & \h^0(Y,\mathcal{M}_1^2) \arrow[d,"g"]\\
    \fil_0\br(K^h)[p] \arrow[r] &\br(K^h)[p] \arrow[r,"\text{res}"] &\h^0(K^h,\br(K^h_{nr})[p]).
\end{tikzcd}
\end{equation}
Note that we can build a diagram analogous to~\eqref{diagram: vanishing cycles} for every finite field extension $L'/L$.

We are left to prove the existence (over a field extension $L'$ of $L$) of an element $\A$ in $\br(X_{L'})[p]$ such that $f_{L'}(\A)\ne 0$. In order to do this, we analyse the spectral sequence of vanishing cycles over an algebraic closure $\bar{L}$ of $L$. Let $\Lambda$ be the integral closure of $\Os_L$ in $\bar{L}$ and $\bar{\ell}$ the residue field of $\Lambda$. Let $\bar{X}$, $\bar{\mathcal{X}}$ and $\bar{Y}$ be the base change of $X,\mathcal{X}$ and $Y$ to $\bar{L},\Lambda$ and $\bar{\ell}$ respectively. We define $\bar{\mathcal{M}}_1^q:=\bar{i}^*\text{R}^q\bar{j}_*\left(\Z/p\Z(1)\right)$, where $\bar{i}$ and $\bar{j}$ denote the inclusion of the generic and special fibre ($\bar{X}$ and $\bar{Y}$, respectively) in $\bar{\mathcal{X}}$. These sheaves build a spectral sequence of vanishing cycles: 
\[
    \bar{E}^{s,t}_2:=\h^s(\bar{Y},\bar{\mathcal{M}}^t_1)\Rightarrow \h^{s+t} \left(\bar{X},\Z/p\Z(1)\right).
\]
We are interested in the map $\bar{f}_1\colon\br(\bar{X})[p]\rightarrow \bar{E}^{0,2}_\infty$, defined in the same way as the map $f_1$.

Combining \cite[Theorem~8.1]{BlochKatoEtale} with the short exact sequence~(8.0.1) and Proposition~7.3 of \cite{BlochKatoEtale} we get that if $Y$ is ordinary, then 
\begin{equation}\label{eq: cohomology of M^q_1 ordinary case}
    \h^n(\bar{Y},\bar{\mathcal{M}}^q_1)\simeq \h^n(\bar{Y},\Omega^q_{\bar{Y},\log}) \, \text{ and } \, \h^n(\bar{Y},\Omega^q_{\bar{Y},\log})\otimes_{\F_p} \bar{\ell}\xrightarrow{\sim}\h^n(\bar{Y},\Omega^q_{\bar{Y}}).
\end{equation}
for every $q,n$.

The assumption on the special fibre $Y$ assures us that the spectral sequence of vanishing cycles $\bar{E}_2^{p,q}$ degenerates at the second page, see \cite[Page 307]{BlochEsnault}. Thus, 
\[
    \bar{E}_\infty^{0,2}=\ker(\bar{E}_2^{0,2}\rightarrow \bar{E}_2^{2,1})=\bar{E}_2^{0,2}
\]
    and hence (by construction) the map $\bar{f}_1\colon \br(\bar{X})[p]\rightarrow \h^0(\bar{Y},\bar{\mathcal{M}}_1^2)$ is surjective. Finally, by equation~\eqref{eq: cohomology of M^q_1 ordinary case} the existence of a non-trivial global $2$-form on $Y$ assures that $\h^0(\bar{Y},\bar{\mathcal{M}}_1^2)$ is not trivial. 

\begin{proof}[Proof of Theorem~\ref{thm: obstruction after field extension}]
Without loss of generality we can assume that $k$ contains a primitive $p$-root of unity $\zeta$. Let $\ip$ be a prime of good ordinary reduction for $V$ of residue characteristic $p$. Let $L$ be the $p$-adic field $k_{\ip}$ and $X$ be the base change of $V$ to $L$. Let $\A\in \br(\bar{X})[p]$ be such that $\bar{f}_1(\A)\ne 0$, then there is a finite field extension $L'/L$  such that $\A$ is defined over $L'$ (i.e. $\A\in \br(X_{L'})[p]$) and $f_{1,L'}(\A)\ne 0$. Namely, by what we have said above $\A\notin \fil_0 \br(X_{L'})[p]$.

Since $\br(\bar{V})[p]\simeq \br({V\times_{k} \overline{k_{\ip}}})[p]$ there exists $\A\in \br(\bar{V})$ such that $\bar{f}_1(\A)\ne 0$. Hence, if $k'/k$ is a finite field extension such that $\A$ is defined over $k'$, we know from what we have shown above that there is a prime $\ip'$ above $\ip$ such that $\ev_\A$ is non-constant on $X_{k'}(k'_{\ip'})$.
\end{proof}

\section{Ordinary case: examples}\label{section: ordinary case examples}
In this section we give examples on K$3$ surfaces defined over number fields, for which primes of good reduction play a role in the Brauer--Manin obstruction. We prove that if $X$ is a Kummer K$3$ surface coming from the product of elliptic curves defined over $\Q$ with good ordinary reduction at the prime $2$ and full $2$-torsion defined over $\Q_2$, then $\br(X)[2]=\Ev_{-1}\br(X)[2]$ (cf. Theorem~\ref{thm: Kummer}). This theorem proves what was already predicted by Ieronymou after some computational evidence, see \cite[Remark~2.6]{ieronymou2023odd}.
\subsection{The case of K3 surfaces}\label{section: the case of K3 surfaces}
For  a smooth variety, we write $\omega_X$ for the canonical sheaf of $X$. 
\begin{defi}
    An {algebraic} K$3$ surface is a smooth projective $2$-dimensional variety over a field $k$ such that $\omega_X\simeq \Os_X$ and $\h^1(X,\Os_X)=0$.\footnote{Here $k$ can be any field.}
\end{defi}

More precisely we work with two kinds of K$3$ surfaces (see \cite[Section~1.1]{Huybrechts} for more details):
\begin{enumerate}
    \item $X$ smooth quartic surface in $\p^3_k$;
    \item $X=\mathrm{Kum}(A)$, where $A$ is an abelian surface over a field of characteristic different from $2$ and $X$ is obtained by resolving singularities from the quotient $A/\iota$, with $\iota \colon A\rightarrow A$ involution given by $x\mapsto -x$.  
\end{enumerate}
The special fibre of a K$3$ surface with good reduction is still a K$3$ surface, see for example \cite[Remark~11.5]{BrightNewton}. We start by stating the following well known result.
\begin{lemma}\label{lemma: criterium ordinary k3}
    Let $p$ be a prime number and $Y$ a K$3$ surface over the finite field $\F_{p^n}$ for some non-negative $n$. Then $Y$ is ordinary if and only if $|Y(\F_{p^n})| \not \equiv 1 \mod p$.
\end{lemma}
\begin{proof}
    The proof is an almost immediate consequence of \cite[Section~1]{BogomolovZarhin}\footnote{For more details see for example \cite{PhDthesis}}.
\end{proof}

Given a K$3$ surface $X$ defined over a local field $L$ of residue characteristic $p$ and such that $(p-1)\mid e$, we are looking for elements in $\br(X)$ such that the corresponding evaluation map is non-constant. Since for K$3$ surfaces $\h^1(\bar{Y},\Z/p\Z)=0$, then by \cite[Lemma~11.3]{BrightNewton} $\fil_0\br(X)=\Ev_0\br(X)=\Ev_{-1}\br(X)$. This means that we are looking for $\A\notin \Ev_0\br(X)$, namely such that there is $n\geq 1$ with $\A\in \fil_n \br(X)$ and $\mathrm{rsw}_{n,\pi}(\A)\ne (0,0)$. 

\begin{lemma}\label{lemma: explicit global 2-form}
    Let $\ell$ be a field and $f(x_0,x_1,x_2,x_3)\in \ell[x_0,x_1,x_2,x_3]$ be a homogeneous polynomial of degree $4$. Assume that the corresponding projective variety $Y$ is smooth. Then $Y$ is a K$3$ surface and the $1$-dimensional $\ell$-vector space of global $2$-forms is generated by the $2$-form
    \[
        \omega=\frac{d\left(\frac{x_1}{x_0}\right)\wedge d\left(\frac{x_2}{x_0}\right)}{\frac{1}{x_0^3}\cdot \frac{\partial f}{\partial x_3}}.
    \]
\end{lemma}
\begin{proof}
    This result is most likely well know. Since we are not able to find a reference in the leterature, we refer to the author PhD thesis were the details of the proof have been included, \cite[Lemma~4.3.3]{PhDthesis}.
\end{proof}


We start by recalling the following example, which is the central result of \cite{Pagano}.
\begin{example}[\cite{Pagano}]\label{example: BM obstruction over Q_2}
Let ${V}\subseteq \p^3_{\Q}$ be the projective K$3$ surface defined by the equation 
\begin{equation}\label{eq}
    x^3y+y^3z+z^3w+w^3x+xyzw=0.
\end{equation}
Then $V$ has good ordinary reduction at $2$ and
the class of the quaternion algebra 
\[
    \A=\left(\frac{z^3+w^2x+xyz}{x^3},-\frac{z}{x}\right)\in \br \, \Q(V)
\]
defines an element in $\br(V)$. The evaluation map $\ev_\A \colon {V}(\Q_2)\rightarrow \br(\Q_2)$ is non-constant, and therefore gives an obstruction to weak approximation on $X$. 

\vspace{2mm}

In this case, $\h^0(Y, \Omega^2_Y)$ is a one-dimensional $\F_2$ vector space. Let $\omega$ be the only non-trivial element, then $C(\omega)=0$ or $C(\omega)=\omega$. However, $Y$ being ordinary implies that $\h^0(Y,B^2_Y)=0$, hence by Lemma~\ref{lemma: (a) global form is log (b) global form is exact} and Corollary~\ref{cor: Exact Forms in terms of Cartier Op} $C(\omega)=\omega$ and $\h^0(Y,\Omega^2_{Y,\log})$ is a one-dimensional $\F_2$-vector space. From Lemma~\ref{lemma: explicit global 2-form} we get that the non-zero global logarithmic $2$-form $\omega$ can be written (locally) as: 
\[
    \omega=\frac{d\left(\frac{z^3+w^2x+xyz}{x^3}\right)}{\left(\frac{z^3+w^2x+xyz}{x^3}\right)}\wedge\frac{d\left(\frac{z}{x}\right)}{\left(\frac{z}{x}\right)}.
\]
If we denote by $f$ and $g$ the functions $\frac{z^3+w^2x+xyz}{x^3}$ and $\frac{z}{x}$ seen as elements in the function field $F$ of $Y$, then we see that the two functions appearing in the definition of $\A$ are lifts to characteristic $0$ of $f$ and $g$, and hence from Proposition~\ref{propGradedPieces}
\[
\rho_0(\omega,0)=\left[\left\{ \frac{z^3+w^2x+xyz}{x^3},-\frac{z}{x}\right\}\right]\in \mathrm{gr}^0.
\]
Using Proposition~\ref{prop: rsw and map rho_m}, 
\[
    \rsw_{2,\pi}(\A)=(\omega,0)\ne (0,0)
\]
 and $\A\notin \fil_1\br(X)[2]\supseteq \Ev_{-1}\br(X)[2]$. 
\end{example}
\subsection{Kummer K3 surfaces over 2-adic fields}

Let $L$ be a $2$-adic field. The recent paper  \cite{LazdaSkoro} of Lazda and Skorobogatov gives a way to detect whether the Kummer K$3$ surface attached to an abelian variety $A/L$ with good (non supersingular) reduction has good (non supersingular) reduction \footnote{This was already known for K$3$ surfaces over $p$-adic fields with $p\ne 2$, for the supersingular case in characteristic $2$, see \cite{matsumoto2023supersingular}.}. In \cite{SkoroZarhin} Skorobogatov and Zarhin link the transcendental part of the Brauer group of a Kummer K$3$ surface to the one of the underlying abelian variety (cf. Section~\ref{subsubsection: Kummer K3 generalities}). Combining these results makes it possible to construct examples of K$3$ surfaces with good reduction at the prime $2$ for which we are able to study the Brauer group.

\vspace{2mm}

In this section, we prove that for every pair of elliptic curves $E_1,E_2$ over $\Q$ with good ordinary reduction at $p=2$ and full $2$-torsion defined over $\Q_2$, the $2$-torsion elements in the Brauer group of the corresponding Kummer K$3$ surface do not play a role in the Brauer--Manin obstruction to weak approximation. In particular, this shows that the field extension in Theorem \ref{thm: obstruction after field extension} is needed. We then use these computations to exhibit an example of a K$3$ surface $X$ over $\Q_2$ with good ordinary reduction and such that $\br(X)=\Ev_{-1}\br(X)$, showing that the converse of Theorem \ref{intro_thm: ordinary good reduction} does not hold in general.   

\subsubsection{Brauer group of Kummer K3 surfaces}\label{subsubsection: Kummer K3 generalities}
Let $A$ be an abelian surface over a field $k$ of characteristic different from $2$ and $V=\mathrm{Kum}(A)$ the corresponding Kummer surface. Skorobogatov and Zarhin \cite{SkoroZarhin} prove that there is a well-defined map
\[
\pi^*\colon \br(V)\rightarrow \br(A)
\]
that induces an injection of $\br(V)/\br_1(V)$ into $\br(A)/\br_1(A)$. They also prove that this injection is an isomorphism on the $p$-torsion for all odd primes $p$, see \cite[Theorem 2.4]{SkoroZarhin}. 
We say that an element $\A\in \br(A)$ \emph{descends} to $\br(V)$ if there exists $\mathcal{C}\in \br(V)$ such that $\pi^*(\mathcal{C})=\A$. 
\begin{lemma}\label{lemma: pi^*(Ev_-1br(V)) is in Ev_-1br(A)}
    Let $V=\mathrm{Kum}(A)$, $\mathcal{C}\in \br(V)$ and $\mathcal{B}:=\pi^*(\mathcal{C})\in \br(A)$. Let $\ip$ be a prime in $\Os_{k}$; if the image of  $\mathcal{C}$ in $\br(V_{\ip})$ lies in $ \Ev_{-1}\br(V_{\ip})$ then the image of $\mathcal{B}$ in $\br(A_{\ip})$ lies in $ \Ev_{-1}\br(A_{\ip})$.
\end{lemma}
\begin{proof}
    The result follows from the definition of $\Ev_{-1}$ and the fact that for any finite field extension $M/k_{\ip}$ and $P\in A(M)$ we have 
    \[
        \ev_{\mathcal{B}}(P)=\ev_{\pi^*(\mathcal{C})}(P)=\ev_{\mathcal{C}}(\pi(P)).
    \]
\end{proof}
In \cite{SkoroZarhin} Skorobogatov and Zarhin show that given two elliptic curves $E_1$ and $E_2$ with Weierstrass equations
\[
E_1: \, v_1^2=u_1\cdot (u_1-\gamma_{1,1})\cdot(u_1-\gamma_{1,2}), \quad E_2: \,v_2^2=u_2\cdot (u_2-\gamma_{2,1}) \cdot (u_2-\gamma_{2,2})
\]
the quotient $\br(E_1\times E_2)[2]/\br_1 (E_1\times E_2)[2]$ is generated by the classes of the four Azumaya algebras 
\[
    \mathcal{A}_{\epsilon_1,\epsilon_2}=((u_1-\epsilon_1)(u_1-\gamma_{1,2}),(u_2-\epsilon_2)(u_2-\gamma_{2,2})) \; 
    \text{ with } \; \epsilon_i\in \{0,\gamma_{i,1}\}. 
\]
Finally, if $M$ is the matrix
    \[
    M=
    \begin{pmatrix}
    1 &\gamma_{1,1}\cdot \gamma_{1,2} &\gamma_{2,1}\cdot \gamma_{2,2} &-\gamma_{1,1}\cdot \gamma_{2,1}\\
    \gamma_{1,1}\cdot \gamma_{1,2} &1 &\gamma_{1,1}\cdot \gamma_{2,1} &\gamma_{2,1}\cdot (\gamma_{2,1}-\gamma_{2,2}) \\
    \gamma_{2,1}\cdot \gamma_{2,2} &\gamma_{1,1}\cdot \gamma_{2,1} &1 &\gamma_{1,1}\cdot(\gamma_{1,1}-\gamma_{1,2})\\
    -\gamma_{1,1}\cdot \gamma_{2,1} &\gamma_{2,1}\cdot (\gamma_{2,1}-\gamma_{2,2}) &\gamma_{1,1}\cdot(\gamma_{1,1}-\gamma_{1,2}) &1
    \end{pmatrix}
    \]
    then by \cite[Lemma 3.6]{SkoroZarhin}:
    \begin{enumerate}
        \item $\A_{\gamma_{1,1},\gamma_{2,1}}$ descends to $\br(V)$ if and only if the entries of the first row of $M$ are all squares;
        \item $\A_{\gamma_{1,1},0}$ descends to $\br(V)$ if and only if the entries of the second row of $M$ are all squares;
        \item $\A_{0,\gamma_{2,1}}$ descends to $\br(V)$ if and only if the entries of the third row of $M$ are all squares;
        \item $\A_{\text{\tiny 0},\text{\tiny 0}}$ descends to $\br(V)$ if and only if the entries of the last row of $M$ are all squares.
    \end{enumerate}

\subsubsection{Product of elliptic curves with good reduction at $2$ and full $2$-torsion}
In order to use the results just summarised we need to analyse what the $2$-torsion points of an elliptic curve with good ordinary reduction at $2$ look like. Let $E/\Q$ be an elliptic curve having good reduction at $2$ and such that the $2$-torsion of $E$ is defined over $\Q_2$, i.e. $E(\Q_2)[2]=E(\bar{\Q}_2)[2]$. Without loss of generality we can assume $E/\Q$ to be defined by the minimal Weierstrass equation\footnote{If $y^2+a_1xy+a_3y=x^3+a_2x^2+a_4 x+a_6$ is a minimal Weierstrass equation over $\Z_2$, then $a_1\in \Z_2^\times$ (since the reduction modulo $2$ is ordinary). The change of variables $(x,y)\mapsto (a_1^2x'-a_3/a_1, a_1^3y'+(1-a_1^5)/(2a_1^4)x')$ gives another minimal Weierstrass equation of the form $y^2+xy=x^3+ax^2+bx+c$.}
\begin{equation}\label{eqOrdinaryEllipticCurve}
    y^2+xy= x^3+ax^2+bx+c
\end{equation}
with $a,b,c\in \Z_2$. Let $\alpha_i,\beta_i\in \Q_2$ be such that 
\[
    E(\Q_2)[2]=\{\mathcal{O},(\alpha_1,\beta_1),(\alpha_2,\beta_2),(\alpha_3,\beta_3)\},
\] 
with $\mathcal{O}$ the point at infinity of $E$. 
\begin{lemma}\label{lemma:2adicValOf2TorsionPoint}
    Assume that $\beta_1,\beta_2,\beta_3$ are ordered as
    \[
    \mathrm{ord}_2(\beta_1)\leq \mathrm{ord}_2(\beta_2) \leq \mathrm{ord}_2(\beta_3).
    \]
    Then $\mathrm{ord}_2(\alpha_1)=-2$ and $\alpha_2,\alpha_3\in \Z_2$. 
\end{lemma}
\begin{proof}
    The $2$-torsion points on $E$ can be computed through the $2$-division polynomial of $E$, which is $\psi_2(x,y)=2y+x$.
In particular, $\alpha_i=-2\beta_i$ with $\beta_i$ solution of 
\begin{equation*}
    \Phi(y):=y^2-2y^2-((-2y)^3+a(-2y)^2+b(-2y)+c).
\end{equation*}
The polynomial $\Phi(y)$ can be rewritten as 
\begin{equation}\label{eq: polynomial phi(y)}
    \Phi(y)=8y^3-(1+4a)y^2+2by-c.
\end{equation}
 Looking at the coefficients of $\Phi(y)$ we get that 
\[
\begin{cases}
    \mathrm{ord}_2(\beta_1+\beta_2+\beta_3)=\mathrm{ord}_2 (1 +4a)-\mathrm{ord}_2(8)\\
    \mathrm{ord}_2(\beta_1\beta_2+\beta_1\beta_3+\beta_2\beta_3)=\mathrm{ord}_2(2b)-\mathrm{ord}_2(8)\\
    \mathrm{ord}_2(\beta_1\beta_2\beta_3)=\mathrm{ord}_2(c)-\mathrm{ord}_2(8).
\end{cases}
\]
From the first equation, we get 
$\mathrm{ord}_2(\beta_1)\leq -3$
that combined with \cite[Theorem~VIII.7.1]{Silverman} tells us that $\mathrm{ord}_2(\beta_1)=-3$. From the third equation, we get that
$\mathrm{ord}_2(\beta_2)+\mathrm{ord}_2(\beta_3)\geq 0.$
If $\mathrm{ord}_2(\beta_2)=\mathrm{ord}_2(\beta_3)$ then we immediately get that both valuations are non negative. Otherwise, $\mathrm{ord}_2(\beta_2)<\mathrm{ord}_2(\beta_3)$ and from the second equation we get $\mathrm{ord}_2(\beta_2)\geq 1$, which implies that also in this case both $\beta_2$ and $\beta_3$ have non negative $2$-adic valuation. The result now follows from the fact that $\alpha_i=-2\beta_i$. 
\end{proof}

\begin{lemma}\label{lemma:changeOfVarElliptic}
    The change of variables given by 
\begin{equation}\label{eqChangeVar}
    \begin{cases}
        u=4x-4\alpha_1\\
        v=4(2y+x)
    \end{cases}
\end{equation}
induces an isomorphism between $E$ and the elliptic curve given by the equation
    \begin{equation}
        v^2=u(u-\gamma_1)(u-\gamma_2)
    \end{equation}
    where $\gamma_1=4\cdot (\alpha_2-\alpha_1)$ and $\gamma_2=4\cdot(\alpha_3-\alpha_1)$.
\end{lemma}
\begin{proof}
    The change of variables 
    \begin{equation*}
        \begin{cases}
            u_1=4x\\
            v_1=4(2y+x)
        \end{cases}
    \end{equation*}
    sends the elliptic curve given by the equation 
    \begin{equation}\label{eqIntermediate}
        v_1^2=u_1^3 + (4a+1) u_1^2 + 16 c
    \end{equation}
    to the elliptic curve given by equation~\eqref{eqOrdinaryEllipticCurve}.
    Moreover, the $2$-division polynomial of $E$ is given by $2y+x$. Hence the non-trivial $2$-torsion points on the elliptic curve given by equation~\eqref{eqIntermediate} are sent to non-trivial $2$-torsion points on $E$. The extra translation $u=u_1-4\alpha_1$ and $v=v_1$ gives the desired equation. 
\end{proof}
Let $E_1$ and $E_2$ be two elliptic curves with equations of the form~\eqref{eqOrdinaryEllipticCurve}. We denote by $(a_i,c_i)$ the parameters that determine the equation attached to $E_i$, by $(\alpha_{i,j},\beta_{i,j})$, $j\in \{1,2,3\}$ the non-trivial $2$-torsion points of $E_i$ and by $A$ the abelian surface given by the product of $E_1$ with $E_2$. We denote by $\langle \Ev_{-1}\br(A)[2], \br_1(A)[2] \rangle$ the subgroup of $\br(A)[2]$ generated by $\Ev_{-1}\br(A)[2]$ and $\br_1(A)[2]$, where $\br_1(A)[2]$ is the algebraic Brauer group of $A$.

\begin{lemma}\label{lemma: A_epsilon_1,epsilon_2 in Ev_-1}
    Assume that $\epsilon_1$ and $\epsilon_2$ are as in Section \ref{subsubsection: Kummer K3 generalities}; then the class of the quaternion algebra $\mathcal{A}_{\epsilon_1,\epsilon_2}$ lies in $\langle \Ev_{-1}\br(A)[2], \br_1(A)[2] \rangle$ if and only if at least one between $\epsilon_1$ and $\epsilon_2$ is different from $0$.
\end{lemma}
\begin{proof}
    We fix $\pi=2$ as a uniformiser and $\xi=-1$ as a primitive $2$-root of unity. We start by assuming that at least one among $\epsilon_1$ and $\epsilon_2$ is different from $0$. By the symmetry of the statement, we can assume without loss of generality that $\epsilon_1\ne 0$. Then 
    \[
    \A_{\epsilon_1,\epsilon_2}=((u_1-\gamma_{1,1})\cdot (u_1-\gamma_{1,2}) \, , \,(u_2-\epsilon_2)\cdot (u_2-\gamma_{2,2}) )=(u_1,f_{\epsilon_2}(u_2))
    \]
    where $f_{\epsilon_2}(u_2)=(u_2-\epsilon_2)\cdot (u_2-\gamma_{2,2})$.
    The quaternion algebra $\A_{\epsilon_1,\epsilon_2}$ corresponds via the change of variables of Lemma \ref{lemma:changeOfVarElliptic} to 
    \[
    \mathcal{A}_{\epsilon_1,\epsilon_2}=(4\cdot(x_1-\alpha_{1,1}) \, , \, f_{\epsilon_2}(4x_2-4\alpha_{2,1}))=(x_1-\alpha_{1,1} \, ,\,f_{\epsilon_2}(4x_2-4\alpha_{2,1})).
    \]
    We define 
    \[
    g_{\epsilon_2}(x_2):=\begin{cases}
        (x_2-\alpha_{2,2})\cdot (x_2-\alpha_{2,3}) \text{ if }\epsilon_2=\gamma_{2,1};\\
        (x_2-\alpha_{2,1})\cdot (x_2-\alpha_{2,3}) \text{ if }\epsilon_2=0.
    \end{cases}
    \]
    Then, $16\cdot g_{\epsilon_2}(x_2)=f_{\epsilon_2}(4x_2-4\alpha_{2,1})$. Thus we can rewrite $\A_{\epsilon_1,\epsilon_2}$ as
    \[
        (-\alpha_{1,1} \, ,\, g_{\epsilon_2}(x_2))\otimes (1+(-\alpha_{1,1}^{-1})\cdot x_1\, ,\, g_{\epsilon_2}(x_2)).
    \]
  Since $(-\alpha_{1,1}\, ,\, g_{\epsilon_2}(x_2))$ lies in $\br_1(A)[2]$, we are left to show that the class of the quaternion algebra $(1+(-\alpha_{1,1})^{-1}\cdot x_1\, ,\, g_{\epsilon_2}(x_2))$ lies in $ \Ev_{-1}\br(A)[2]$. By Lemma \ref{lemma:2adicValOf2TorsionPoint} we know that $\mathrm{ord}_2(\alpha_{1,1}^{-1})=2$ and therefore by Proposition \ref{prop: rsw and map rho_m}
  \[
    (1+(-\alpha_{1,1}^{-1})\cdot  x_1 \, , \, g_{\epsilon_2}(x_2))\in \fil_{0}\br(A)[2].
  \]
  By \cite[Theorem~C]{BrightNewton} in order to establish whether  $(1+(-\alpha_{1,1})^{-1}\cdot  x_1 \, , \, g_{\epsilon_2}(x_2))$ belongs to $\Ev_{-1}\br(A)[2]$ we need to compute $\partial(1+(-\alpha_{1,1}^{-1})\cdot  x_1 \, , \, g_{\epsilon_2}(x_2))$. We have that $g_{\epsilon_2}(x_2)\not \equiv 0 \mod 2$ and from Proposition~\ref{propGradedPieces 1} together with Proposition~\ref{prop: rsw and map rho_m} we get
  \[
  \lambda_\pi \left(\bar{x}_1\cdot d \mathrm{log}(\bar{g}_{\epsilon_2}(\bar{x}_2)),0\right)=\left(1+(-\alpha_{1,1}^{-1}) \cdot x_1 \, , \, g_{\epsilon_2}(x_2)\right )
  \]
  since $1+(-\alpha_{1,1}^{-1})\cdot x_1=1+4 \cdot (s^{-1} \cdot x_1)$ with $s=-4\cdot \alpha_{1,1}\in \Z_2^\times$ and hence $s^{-1}\cdot x_1$ is a lift to characteristic $0$ of $\bar{x}_1$. 
  Therefore, by definition of the residue map $\partial$, we get
  \[
  \partial((1+(-\alpha_{1,1}^{-1}) \cdot x_1 \, ,\,g_{\epsilon_2}(x_2)))=0
  \]
    which by Theorem \ref{thm: evaluation fil and refined Swan conductor} implies that 
    \[
        (1+(-\alpha_{1,1}^{-1}) \cdot x_1 \, , \, g_{\epsilon_2}(x_2))\in \Ev_{-2}\br(A)[2]\subseteq \Ev_{-1}\br(A)[2].
    \]
    In order to end the proof we are left to show that $\A_{\text{\tiny 0},\text{\tiny 0}}\notin \langle \Ev_{-1}\br(A)[2],\br_1(A)[2]\rangle$.
    The change of variables of Lemma \ref{lemma:changeOfVarElliptic} sends the class of the quaternion algebra 
    \[
        \A_{\text{\tiny 0},\text{\tiny 0}}=(u_1\cdot (u_1-\gamma_{1,2})\, , \, u_2\cdot(u_2-\gamma_{2,2}))=(u_1-\gamma_{1,1} \, ,\, u_2-\gamma_{2,1})
    \]
    to the class of the quaternion algebra
    \[
    (4\cdot(x_1-\alpha_{1,2}) \, ,\, 4\cdot(x_2-\alpha_{2,2}))=(x_1+2\beta_{1,2} \, , \, x_2+2\beta_{2,2}).
    \]
    From Proposition~\ref{propGradedPieces}(d) the latter is such that  
    \[
    \rho_0\left(\frac{d(\bar{x}_1)}{\bar{x}_1}\wedge \frac{d(\bar{x}_2)}{\bar{x}_2}\right)=\left[\left\{x_1+2\beta_{2,1},x_2+2\beta_{2,2}\right\}\right]\in \mathrm{gr}^0.
    \]
    In fact, $x_1+2\beta_{2,1},x_2+2\beta_{2,2}$ and $x_2+2\beta_{2,2}$ are lifts to characteristic $0$ of $\bar{x}_1$ and $\bar{x}_2$ respectively. Note that, $\frac{d(\bar{x}_1)}{\bar{x}_1}\wedge \frac{d(\bar{x}_2)}{\bar{x}_2}$ comes from a global $2$-form on the special fibre $Y$ of $A$ and hence it is non-zero in its function field. Finally, using Proposition~\ref{prop: rsw and map rho_m} we get that 
    \[
        \rsw_{2,\pi}((x_1+2\beta_{2,1},x_2+2\beta_{2,2}))=\left(\frac{d(\bar{x}_1)}{\bar{x}_1}\wedge \frac{d(\bar{x}_2)}{\bar{x}_2},0\right)\ne(0,0)
    \]
    and hence $(x_1+2\beta_{1,2} , x_2+2\beta_{2,2})\notin \fil_1\br(A)[2]\supseteq \Ev_{-1}\br(A)[2]$. Moreover, as a consequence of Corollary~\ref{cor: refined Swan conductor and algebraic elements} we get that $\A_{\text{\tiny 0},\text{\tiny 0}}\notin \langle \Ev_{-1}\br(A)[2],\br_1(A)[2]\rangle$. In fact, otherwise, there would be an element $\A_1\in \Ev_{-1}\br(A)[2]$ such that $\A_{\text{\tiny 0},\text{\tiny 0}}\otimes \A_1 \in \br_1(A)[2]$, but $\A_{\text{\tiny 0},\text{\tiny 0}}\otimes \A_1$ has the same refined Swan conductor as $\A_{\text{\tiny 0},\text{\tiny 0}}$.
\end{proof}
\begin{rmk}\label{rmk: lemma A_epsilon_1,epsilon_2 in Ev_-1 with base change}
    We will later use a slightly stronger version of the lemma above. Let $L/\Q_2$ be any field extension. Then, the image of $\A_{\epsilon_1,\epsilon_2}$ in $\br(A_L)[2]$ lies in $\langle \Ev_{-1}\br(A_L)[2],\br_1(A_L)[2]\rangle$ if and only if at least one among $\epsilon_1$ and $\epsilon_2$ is different from $0$. 
    
    We clearly have that $\A_{\epsilon_1,\epsilon_2} $ in $\langle \Ev_{-1}\br(A)[2],\br_1(A)[2]\rangle$ implies $\mathrm{res}(\A_{\epsilon_1,\epsilon_2})$ in $\langle \Ev_{-1}\br(A_L)[2],\br_1(A_L)[2]\rangle$. Moreover, we have proven that the first component of $\rsw_{2,\pi}(\A_{\text{\tiny 0},\text{\tiny 0}})$ is different from $0$, and hence using Lemma~\ref{lemma: base change L'/L} we get that $\rsw_{e_{L'/L}2,\pi}(\mathrm{res}(\A_{\text{\tiny 0},\text{\tiny 0}}))\ne (0,0)$, hence the image of $\A_{\text{\tiny 0},\text{\tiny 0}}$ in $\br(A_L)[2]$ does not lie in $\langle \Ev_{-1}\br(A_L)[2],\br_1(A_L)[2]\rangle$
\end{rmk}
\subsubsection{Brauer--Manin obstruction from 2-torsion elements}
Let $V$ be $\mathrm{Kum}(A_{\Q})$. We denote by $A$ and $X$ the base changes of the abelian surface $A_{\Q}$ and the corresponding Kummer surface $X$ to $\Q_2$. By Section \ref{subsubsection: Kummer K3 generalities} we know that $\A_{\text{\tiny 0},\text{\tiny 0}}$ descends to $X$ if and only if 
\[
 \left[ -\gamma_{1,1}\cdot \gamma_{2,1}, \gamma_{2,1}(\gamma_{2,1}-\gamma_{2,2}), \gamma_{1,1} (\gamma_{1,1}-\gamma_{1,2}),1\right] \in (\Q_2^{\times 2})^4.
\]
By construction, $\gamma_{1,1}=4\cdot(\alpha_{1,2}-\alpha_{1,1})=8\beta_{1,1}-8\beta_{1,2}$ and therefore
\[
\gamma_{1,1}\equiv 8\beta_{1,1} \equiv 1+4a_1 \mod 8.
\]
In fact, from Lemma \ref{lemma:2adicValOf2TorsionPoint} and more precisely from equation~\eqref{eq: polynomial phi(y)} we know that $8\beta_{1,1}+8\beta_{1,2}+8\beta_{1,3}=1+4a_1$ and $\mathrm{ord}_2(\beta_{1,2})=\mathrm{ord}_2(\beta_{1,3})\geq 0$.
Similarly, 
\[
\gamma_{2,1}\equiv 8\beta_{2,1} \equiv 1+4a_2 \mod 8.
\]
In particular, both $\gamma_{1,1}$ and $\gamma_{2,1}$ are either $1$ or $5$ modulo $8$; hence $-\gamma_{1,1}\cdot \gamma_{2,1}$ is either $-1$ or $3$ and therefore it is never a square. 

\vspace{1mm}

In summary, we just showed that $\A_{\text{\tiny 0},\text{\tiny 0}}\in \br(A)$ never descends to $\br(X)$. We are ready to prove the main theorem of this section.
\begin{thm}\label{thm: Kummer}
    Let $X=\mathrm{Kum}(A)$, where $A=E_1\times E_2$ is as in Section \ref{subsubsection: Kummer K3 generalities}. Then $\br(X)[2]=\Ev_{-1}\br(X)[2]$.
\end{thm}
\begin{proof}
We recall that if $\A_{\epsilon_1,\epsilon_2}$ descends to $\br(X)[2]$, we denote by $\mathcal{C}_{\epsilon_1,\epsilon_2}$ the corresponding element in $\br(X)[2]$, i.e. $\mathcal{C}_{\epsilon_1,\epsilon_2}$ is such that $\pi^*(\mathcal{C}_{\epsilon_1,\epsilon_2})=\A_{\epsilon_1,\epsilon_2}$. We need to prove that if $\A_{\epsilon_1,\epsilon_2}$ descends to $\br(X)$, then $\mathcal{C}_{\epsilon_1,\epsilon_2}$ lies in $\Ev_{-1}\br(X)$. Since we have already shown at the beginning of this section that $\A_{\text{\tiny 0},\text{\tiny 0}}$ never descends to $\br(X)$ we are left to show it for $(\epsilon_1,\epsilon_2)\ne (0,0)$.
   
Let $L/\Q_2$ be such that all elements appearing in the matrix $M$ of Section~\ref{subsubsection: Kummer K3 generalities} are squares, i.e. the injective map 
\[
    \pi^*\colon \frac{\br(X)[2]}{\br_1(X)[2]}\hookrightarrow \frac{\br(A)[2]}{\br_1(A)[2]}
\]
is an isomorphism. 

With abuse of notation, we denote by $\mathrm{res}$ both 
\[
\mathrm{res}\colon \br(A)\rightarrow \br(A_L) \quad \text{and} \quad \mathrm{res}\colon \br(X)\rightarrow \br(X_L).
\]
We denote by $(\mathcal{C}_{\epsilon_1,\epsilon_2})_L$ the pre-image of $\mathrm{res}(\A_{\epsilon_1,\epsilon_2})\in \br(A_L)[2]$.

     From \cite[Theorem~2]{LazdaSkoro} we know that the reduction of $X$ is an ordinary K$3$ surface. Let $e$ be the ramification index of $L/\Q_2$ and $\pi_L$ a uniformiser of $\Os_L$; then $\fil_n \br(X_L)[p]=\fil_0\br(X_L)[p]=\Ev_{-1}\br(X_L)[p]$ if $n<e':=2e$ and $\fil_{e'}\br(X)[p]=\br(X_L)[p]$, see Remark~\ref{cor: on the image of the refined swan conductor 1}. Hence, using Corollary~\ref{cor: on the image of the refined swan conductor 3}  we have an injection 
    \[
        \mathrm{mult}_{\overline{c}}\circ \rsw_{e',\pi_L}:\frac{\br(X_L)[2]}{\Ev_{-1}\br(X_L)[2]} \hookrightarrow \h^0(Y_\ell,\Omega^2_{Y_\ell,\log})
    \]
    where $\ell$ is the residue field of $L$.
    
    Since $X$ (and hence all its base changes) has good ordinary reduction, we know that $\h^0(\bar{Y},\Omega^2_{\bar{Y},\log})\otimes_{\F_2}\bar{\ell}$ is a one-dimensional $\bar{\ell}$-vector space (cf. equation~\eqref{eq: cohomology of M^q_1 ordinary case}) and hence $\br(X_L)[2]/\Ev_{-1}\br(X_L)[2]$ is a vector space of dimension at most $1$ over $\F_2$.
    
     From Lemma \ref{lemma: A_epsilon_1,epsilon_2 in Ev_-1} we know that $\A_{\text{\tiny 0},\text{\tiny 0}}\notin \fil_1\br(A)[2]$. Applying Lemma~\ref{lemma: base change L'/L} we get that $\mathrm{res}(\A_{\text{\tiny 0},\text{\tiny 0}})\notin \fil_{e}\br(A_L)[2]$ and by Lemma \ref{lemma: pi^*(Ev_-1br(V)) is in Ev_-1br(A)} $(\mathcal{C}_{\text{\tiny 0},\text{\tiny 0}})_L\notin \Ev_{-1}\br(X_L)[2]$. Therefore \[
     \langle [(\mathcal{C}_{\text{\tiny 0},\text{\tiny 0}})_L] \rangle =\frac{\br(X_L)[2]}{\Ev_{-1}\br(X_L)[2]}.
     \] 
Assume that there exists $(\epsilon_1,\epsilon_2)\ne (0,0)$ such that $\A_{\epsilon_1,\epsilon_2}$ descends to $\br(X)$ and the corresponding element $\mathcal{C}_{\epsilon_1,\epsilon_2}$ does not lie in $\Ev_{-1}\br(X)[2]$. 
By Lemma \ref{lemma: base change L'/L} $\mathrm{res}(\mathcal{C}_{\epsilon_1,\epsilon_2})$ does not lie in $\Ev_{-1}\br(X_L)[2]$ and therefore, since the quotient $\br(X_L)[2]/\Ev_{-1}\br(X_L)[2]$ is a $1$-dimensional $\F_2$-vector space, we have that also the product $\mathrm{res}(\mathcal{C}_{\epsilon_1,\epsilon_2})\otimes (\mathcal{C}_{\text{\tiny 0},\text{\tiny 0}})_L$ lies in $\Ev_{-1}\br(X)[2]$. This implies that also the corresponding element in $\br(A_L)[2]$, $\mathrm{res}(\A_{\epsilon_1,\epsilon_2}\otimes \A_{\text{\tiny 0},\text{\tiny 0}})$ lies in $\Ev_{-1}\br(A_L)$. However, since by Remark~\ref{rmk: lemma A_epsilon_1,epsilon_2 in Ev_-1 with base change}  $\mathrm{res}(A_{\epsilon_1,\epsilon_2})$ lies in $ \langle \Ev_{-1}\br(X_L)[2], \br_1(X_L)[2]\rangle$, we get that also $\mathrm{res}(\A_{\text{\tiny 0},\text{\tiny 0}})$ has to lie in $\langle \Ev_{-1}\br(A_L)[2],\br_1(A_L)[2]\rangle$ which gives us the desired contradiction. 
\end{proof}
We end this section with an example of a K$3$ surface over $\Q$ with good ordinary reduction at $2$ and such that $\br(X)=\Ev_{-1}\br(X)$. The existence of such an example shows that the converse of Theorem \ref{intro_thm: ordinary good reduction} is not true, i.e. it is not enough to have that $p-1\mid e$ in order to find an element in $\br(X)$ that does not lie in $\Ev_{-1}\br(X)$.
\begin{example}\label{example: Kummer}
Let $A=E\times E$, where $E$ is the elliptic curve given by the minimal Weierstrass equation 
\[
y^2+xy=x^3-3x^2-4x+11.
\]
With the same notation as in the previous sections, we get $\beta_1=-11/8$, $\beta_2=-1$ and $\beta_3=1$. Hence 
\[
\alpha_1=11/4, \quad \alpha_2=2, \quad \alpha_3=-2 \quad \text{ and } \quad \gamma_{1}=-3, \quad \gamma_2=-19.
\]
The matrix $M$ is of the form
\[
    \begin{pmatrix}
    1 &3\cdot 19 &3\cdot 19 &-9\\
    3\cdot 19 &1 &9 &-3\cdot 18 \\
    3\cdot 19 &9 &1 &-3\cdot 18\\
    -9 &-3\cdot 18 &-3\cdot 18 &1
    \end{pmatrix}.
    \]
    In particular, all the rows of $M$ have at least one term which does not lie in $\Q_2^{\times 2}$. Moreover, using \cite[Proposition~3.7]{SkoroZarhin} we can compute the dimension as an $\F_2$-vector space of the quotient of $\br(X)[2]$ by $\br(\Q_2)[2]$ and in this case
    \[
        \mathrm{dim}_{\F_2}\left(\frac{\br(X)[2]}{\br(\Q_2)[2]}\right)=0.
    \]
    We want to show that $\br(X)\{2\}=\br(\Q_2)\{2\}$. We work by induction on $n$; let $\A$ be in $\br(X)[2^n]$, then 
    \[
        \A^{\otimes 2^{n-1}}\in \br(X)[2]= \br(\Q_2)[2].
    \]
    In particular, given $P\in X(\Q_2)$, we have that
    \[
    (\A \otimes \ev_{\A}(P))^{\otimes 2^{n-1}}=\A^{\otimes 2^{n-1}}\otimes \ev_{\A^{\otimes 2^{n-1}}}(P)=0,
    \]
    hence $\A\otimes \ev_{\A}(P)\in \br(X)[2^{n-1}]$, and by induction hypothesis $\A\in \br(\Q_2)[2^n]$.
\end{example}
\subsection{Examples of Brauer--Manin obstruction}
We continue by giving new examples of primes of good reduction that play a role in the Brauer--Manin obstruction to weak approximation in the case $p=3$ and $p=5$.
\begin{example}
    Let $L=\Q_3(\zeta)$ with $\zeta$ a primitive $3$-root of unity. Let $\pi$ be a uniformiser of $\Os_L$; then $e=e(L/\Q_3)=2$ and the residue field $\ell$ is equal to $\F_3$.

    We define $X$ to be the Kummer K$3$ surface over $L$ attached to the abelian surface $A=E\times E$, with $E$ the elliptic curve over $L$ defined by the Weierstrass equation
    \[
        y^2=x^3+4\cdot x^2+3\cdot x+1.
    \]
    The elliptic curve $E$ (and hence $A$ and $X$) has good ordinary reduction at the prime $\ip=(\pi)$. We denote by $\{x,y,z\}$ and $\{u,v,w\}$ the variables corresponding to the embedding of respectively the first and the second copy of $E$ in $\p^2_L$. We define the cyclic algebra
    \[
    \A:=\left(\frac{v-u}{w},\frac{y-x}{z}\right)_{\zeta} \in \br(L(A))[3].
    \]
    \emph{Claim 1:} $\A$ belongs to $\br(A)[3]$.  
    \begin{proof}
        First of all, notice that from $y^2z=x^3+4x^2z+3xz^2+z^3$ we get
    \[
        z(y-x)(y+x)=(x+z)^3 \quad \text{and} \quad z(y^2-4x^2-3xz-z^2)=x^3.
    \]
    Then:
    \begin{itemize}
        \item[-] if $z=0$, then $x=0$ and therefore $y^2-4x^2-3xz-z^2\ne 0$ and $x\ne y$ and from the equation above, $\A$ is equivalent to 
        \[
        \left(\frac{v-u}{(v^2-4u^2-3uw-w^2)^{-1}},\frac{y-x}{(y^2-4x^2-3xz-z^2)^{-1}}\right)_{\zeta};
        \]
        \item[-] if $x=y$, then $x\ne -y$ and $z\ne 0$ (since $z=0$ implies $x\ne y$) and from the equation above, $\A$ is equivalent to 
        \[
            \left(\frac{w}{v+u},\frac{z}{x+y}\right)_{\zeta}.
        \]
    \end{itemize}  
    Thus we see that, given a divisor $D$ over which $\A$ is not well defined, we are able to find an Azumaya algebra which is equivalent to $\A$ on an open set and well defined along $D$. Hence, $\A$ defines an element in $\br(A)[3]$.
    \end{proof}
    \noindent \emph{Claim 2:} $\A$ does not lie in $\Ev_{-1}\br(A)[3]$.
    \begin{proof}
        The regular global $1$-form on the reduction of $E$ modulo $\ip$ is given by the (local) formula
    \[
        \frac{dx}{2y}=-\frac{1}{2}\cdot \frac{dx\cdot(x-y)}{y(y-x)}=-\frac{1}{2}\cdot \frac{dx \cdot \frac{x}{y}-dx}{y-x}=\frac{d(y-x)}{y-x}
    \]
    where the last equality follows from the fact that on the special fibre $\frac{dx}{y}=\frac{dy}{x}$ and, since we are in characteristic $3$, $\frac{1}{2}=-1$.
    Hence, if we denote by $Y$ the reduction modulo $\ip$ of $A$, we have that the global $2$-form on $Y$ is given by
    \[
    \omega=\frac{d\left(\frac{v-u}{w}\right)}{\left(\frac{v-u}{w}\right)}\wedge \frac{d\left(\frac{y-x}{z}\right)}{\left(\frac{y-x}{z}\right)}.
    \]
    
    Finally, $\rho_0(\omega,0)=\left[\left\{\frac{v-u}{w},\frac{y-x}{z}\right\}\right]$ and hence again by Proposition~\ref{prop: rsw and map rho_m} we get that
    \[
        \rsw_{3,\pi}(\A)\ne (\bar{c}^{-1}\cdot \omega,0)
    \]
    and therefore $\A\notin \fil_2\br(A)[3]\supseteq \Ev_{-1}\br(A)[3]$.
    \end{proof}
    \noindent \emph{Claim 3:} The cyclic algebra $\A$ in $\br(A)[3]$ is not algebraic, i.e. $\A\notin \br_1(A)[3]$.
    \begin{proof}
        If follows directly from Corollary~\ref{cor: refined Swan conductor and algebraic elements}.
    \end{proof}
    Finally, from \cite[Theorem~2.4]{SkoroZarhin} the map 
    \[
    \pi^*\colon \frac{\br(X)[3]}{\br_1(X)[3]}\hookrightarrow \frac{\br(A)[3]}{\br_1(A)[3]}
    \]
    is an isomorphism. Let $\mathcal{B}\in \br(X)[3]$ be such that $\pi^*(\mathcal{B})=\A\in \br(A)[3]$. Then, from Lemma~\ref{lemma: pi^*(Ev_-1br(V)) is in Ev_-1br(A)} we get that $\mathcal{B}\notin \Ev_{-1}\br(X)[3]$, namely (since for K$3$ surfaces $\Ev_0\br(X)=\Ev_{-1}\br(X)$, see Section~\ref{section: the case of K3 surfaces}) the corresponding evaluation map on $X(L)$ is non-constant.
\end{example}

\begin{example}\label{ex: diagonal quartic ordinary}
    Let $X$ be the diagonal quartic surface over $\Q_5$ defined by the equation:
    \[
    5x^4-4y^4=z^4+w^4.
    \]
    Skorobogatov and Ieronymou prove \cite[Theorem~1.1]{IeronymouSkoro}, \cite[Proposition~5.12]{IeronymouSkoro} that there exists an element $\A\in \br(X)[5]$ with surjective evaluation map. Let $L=\Q_5(\sqrt[4]{5})$, $e(L/\Q_5)=4$ and $\alpha\in L$ be such that $\alpha^4=5$. Then the change of variables:
    \[
    (x,y,z,w)\mapsto \left(\frac{x_1}{\alpha},y_1,z_1,w_1\right)
    \]
    sends $X_{L}$ to the diagonal quartic $\tilde{X}/L$ given by the equation
    \[
        x_1^4-4y_1^4=z_1^4+w_1^4.
    \]
    The surface $\tilde{X}$ has good ordinary reduction over $L$, indeed its special fibre is given by the Fermat quartic over $\mathbb{F}_5$, which using Lemma~\ref{lemma: criterium ordinary k3} can be checked to be ordinary. Finally, by Lemma \ref{lemma: base change no good-reduction} we know that $\mathrm{res}(\A)\in \br(X_{L'})=\br(\tilde{X})$ has non-constant evaluation map.
    \begin{rmk}
        We have that $e'=p=5$ and from Corollary~\ref{cor: on the image of the refined swan conductor} we know that $\A$ has to lie in $\fil_p\br(X)[5]\setminus \fil_{p-1}\br(X)[5]$. In fact, since $X$ is ordinary, we have $\h^0(Y,B^2_Y)=\h^0(Y,B^1_Y)=0$ and therefore the image of $\rsw_{n,\pi}$ is zero on the $5$-torsion elements for $n<e'=5$.
    \end{rmk}
\end{example}
\section{Non-ordinary good reduction}\label{Section: non-ordinary good reduction}
In this section we analyse what happens for primes of good non-ordinary reduction in the case of K$3$ surfaces. We are going to prove the following theorem.
\begin{thm}\label{thm: good non-ordinary reduction}
    Let $V$ be a K$3$ surface and $\ip$ be a prime of good non-ordinary reduction for $V$ with $e_{\ip} \leq (p-1)$. Then the prime $\ip$ does not play a role in the Brauer--Manin obstruction to weak approximation on $V$.
\end{thm}
Like in the previous section, let $X$ be the base change to $k_{\ip}$ of $V$, then we denote by $\mathcal{X}$ the base change of $\mathcal{V}$ to $\Os_{\ip}$ and by $Y$ the special fibre $\mathcal{V}(\ip)$. As already explained before, it is enough to show that for every element $\A\in \br(X)$ the evaluation map $\ev_\A\colon X(k_{\ip})\rightarrow \Q/\Z$ is constant. We begin with the following lemma.

\begin{lemma}
    Assume that the special fibre $Y$ of the $\Os_L$-model $\mathcal{X}$ is such that the Cartier operator acts trivially on $\h^0(Y,Z^2_Y)$ and there are no global $1$-forms. Then, if the absolute ramification index $e\leq p-1$, we have $\br(X)=\fil_0 \br(X)$.
\end{lemma}
\begin{proof}
    If $e<p-1$ the result is proven by Bright and Newton, \cite[Lemma~11.2]{BrightNewton} without any extra assumption on $Y$. If $e=p-1$, let $\A\in \fil_n \br(X)$ with $\rsw_{n,\pi}=(\alpha,\beta)$. In order to prove the Lemma it is enough to show that $(\alpha,\beta)=(0,0)$.
    \begin{itemize}
        \item[-] If $n<e'$ then $p\nmid n$ and hence by Corollary \ref{cor: on the image of the refined swan conductor 1} and the assumption that $\h^0(Y,\Omega^1_Y)=0$ we get $\rsw_{n,\pi}(\A)=(0,0)$.
        \item[-] If $n>e'$ then $p\nmid n$ or $p\nmid n-e$ and hence either $\rsw_{n,\pi}(\A)$ or $\rsw_{n-e,\pi}(\A^{\otimes p})$ is zero and therefore by Corollary \ref{cor: on the image of the refined swan conductor 2} $\rsw_{n,\pi}(\A)=(0,0)$. 
        \item[-] If $n=e'$, then by Corollary \ref{cor: on the image of the refined swan conductor 3} we have 
        \[
            \rsw_{e'-e,\pi}(\A^{\otimes p})=(C(\alpha)-\bar{u} \alpha,C(\beta)-\bar{u} \beta).
        \]
        However, by assumption, on $Y$, $C(\alpha)=0$. Hence, 
        \[
            \rsw_{e'-e,\pi}(\A^{\otimes p})=(-\bar{u} \alpha,0).
        \]
        Finally, $e'-e<e'$ and hence $\rsw_{e'-e,\pi}(\A^{\otimes p})=0$, which implies that $\alpha=0$ and thus also in this case $\rsw_{e',\pi}(\A)=(0,0)$.
    \end{itemize} 
\end{proof}
\begin{proof}[Proof of Theorem~\ref{thm: good non-ordinary reduction}]
    If $V$ is a K$3$ surface and $\ip$ is a prime of good non-ordinary reduction, then the special fibre at $\ip$ is a K$3$ surface having good non-ordinary reduction. It is proven in \cite[Proof of Corollary~9.5]{vanderGeerKatsura} that the natural inclusion $B^2_Y\hookrightarrow Z^2_Y$ induces an isomorphism $\h^0(Y,B^2_Y)\simeq \h^0(Y,Z^2_Y)$. The short exact sequence 
    \[
        0\rightarrow B^2_Y \rightarrow Z_Y^2 \xrightarrow{C} \Omega^2_Y \rightarrow 0
    \]
    implies that the Cartier operator is trivial on $\h^0(Y,Z^2_Y)$. Hence, we are in the setting of the previous lemma. The result now follows from the fact that for K$3$ surfaces $\fil_0\br(V_{\ip})=\Ev_{-1}\br(V_{\ip})$.
\end{proof}  

\subsection{An example}\label{subsection: non ordinary case Example}
This example was suggested by Evis Ieronymou. 

In \cite{ErrataIeronymouSkoro} Skorobogatov and Ieronymou prove that for the diagonal quartic surface $X/\Q_3$ defined by the equation
\begin{equation}
    x^4-y^4=12\cdot z^4-9\cdot w^4
\end{equation}
there exists an element $\A \in \br(X)[3]$ such that the corresponding evaluation map is non-constant on $X(\Q_3)$. In analogy with Example~\ref{ex: diagonal quartic ordinary} we have that $X$ has potentially good reduction. Let $L=\Q_3(\zeta_4,\alpha)$, with $\alpha^4=3$, then $e(L/\Q_3)=4$; the change of variables 
\[
    (x,y,z,w)\mapsto \left(x_1,y_1,\frac{z_1}{\alpha},\frac{w_1}{\alpha^2}\right)
\]
sends $X_L$ to the diagonal quartic $\tilde{X}/L$ given by the equation
\[
    x_1^4-y_1^4=4\cdot z_1^4-w_1^4.
\]
Using Lemma~\ref{lemma: criterium ordinary k3} it is possible to check that the K$3$ surface $\tilde{X}$ has good non-ordinary reduction over $L$. The special fibre is the Fermat's quartic over $\mathbb{F}_3$, which by a result of Shioda \cite{Shioda} is known to be supersingular. By Lemma~\ref{lemma: base change no good-reduction} we know that $\mathrm{res}(\A)\in \br(X_L)=\br(\tilde{X})$ has non-constant evaluation map. Hence, $\tilde{X}$ is a K$3$ surface over a $3$-adic field having ramification index $4$, good non-ordinary reduction and such that $\Ev_{-1}\br(\tilde{X})[3]\subsetneq \br(\tilde{X})[3]$. Moreover, $e'=6$ and we proved that $\rsw_{6,\pi}=0$ on the $3$-torsion of $\br(\tilde{X})$ and hence $\mathrm{res}(\A)$ has to lie in $\fil_3\br(\tilde{X})[3]$ and be such that $\rsw_{3,\pi}(\mathrm{res}(\A))\ne (0,0)$.
\section{A family of examples}\label{Section: family of examples}
We end this paper by giving an example of a family of K$3$ surfaces.

Let $\alpha \in \bar{\Q}$ be such that $\alpha^2\in \Z$ and let $V_\alpha$ be the K$3$ surface over $k:=\Q(\alpha)$ defined by the equation 
\begin{equation}\label{eq: example V_a}
        x^3y + y^3 z+ z^3 w- w^4+\alpha^2\cdot  xyzw -2\cdot \alpha^{-1} \cdot x z w^2=0.
    \end{equation}
    \begin{lemma}
        The class of the quaternion algebra 
        \[
         \A:= \left(\frac{z^2+\alpha^2 \cdot  xy}{z^2},-\frac{z}{x}\right)\in \br(k(V_\alpha))
        \]
        lies in $\br(V_\alpha)[2]$. 
    \end{lemma}
\begin{proof}
    Let $f:=z^2+\alpha^2 xy$ and $C_x$,$C_z$,$C_{f}$ be the closed subsets of $V_\alpha$ defined by the equations $x=0$, $z=0$ and $f=0$ respectively. The quaternion algebra $\A$ defines an element in $\br(U)$, where $U:=V_\alpha\setminus (C_x\cup C_z \cup C_{f})$.
The purity theorem for the Brauer group \cite[ Theorem 3.7.2]{BGgroupTheleneSkoro}, assures the existence of the exact sequence 
\begin{equation}\label{purity}
    0 \rightarrow \br(V_\alpha)[2] \rightarrow \br(U)[2] \xrightarrow{\oplus \partial_{D}} \bigoplus_{D} \h^1(k(D),\Z/2)
\end{equation}
where $D$ ranges over the irreducible divisors of $V_\alpha$ with support in $X\setminus U$ and $k(D)$ denotes the residue field at the generic point of $D$. 

In order to use the exact sequence (\ref{purity}) we need to understand what the prime divisors of $V_\alpha$ with support in $V_\alpha\setminus U=C_x \cup C_z \cup C_{f}$ look like. Using MAGMA \cite{magma} it is possible to check the following:
\begin{itemize}
    \item $C_x$ has one irreducible component $D_1$ defined by the set of equations $\{x=0,y^3z+z^3w+w^4=0\}$;
    \item $C_z$ has one irreducible component $D_2$, defined by the set of equations $\{z=0,x^3y-w^4=0\}$;
    \item $C_{f}$ has one irreducible component $D_3$, defined by the set of equations $\{\alpha^2xy + z^2=0, x^3z^2 + y^2z^3 + 2\alpha x^2zw^2 + \alpha^2 xw^4=0,
        \alpha^2 y^3z -x^2z^2 - 2\alpha xzw^2 -\alpha^2 w^4=0
        \}$.
\end{itemize}
Therefore, we can rewrite (\ref{purity}) in the following way:
\begin{equation}\label{purity2}
    0 \rightarrow \br(V_\alpha)[2] \rightarrow \br(U)[2] \xrightarrow{\oplus \partial_{D_i}} \bigoplus_{i=1}^3 H^1(k(D_i),\Z/2).
\end{equation}
Moreover, we have an explicit description of the residue map on quaternion algebras: for an element $(a,b)\in \br(U)[2]$ we have 
\begin{equation}\label{eqResidueMap}
    \partial_{D_i}(a,b)=\left[(-1)^{\nu_i(a)\nu_i(b)}\frac{a^{\nu_i(b)}}{b^{\nu_i(a)}}\right]\in \frac{k(D_i)^\times}{k(D_i)^{\times2}}\simeq H^1(k(D_i),\Z/2) 
\end{equation}
where $\nu_i$ is the valuation associated to the prime divisor $D_i$. This follows from the definition of the tame symbols in Milnor $K$-theory together with the compatibility of the residue map with the tame symbols given by the Galois symbols (see \cite{GilleSzamuely}, Proposition $7.5.1$).

We can proceed with the computation of the residue maps $\partial_{D_i}$ for $i=1,\dots,3$:

\begin{enumerate}
    \item $\nu_1(x)=1$ and $\nu_1(f)=\nu_1(z)=0$. Hence,
    \[
         \partial_{D_1}\left(\frac{f}{z^2},-\frac{z}{x}\right)=\left[\left(\frac{f}{z^2}\right)^{-1}\right]=1\in \frac{k(D_1)^\times}{k(D_1)^{\times 2}}
    \]
    where the last equality follows from the fact that $x=0$ on $D_1$, thus $f\mid_{D_1}=z^2$.
    \item $\nu_2(z)=1$ and $\nu_2(f)=\nu_2(x)=0.$ Hence,
    \[
        \partial_{D_2}\left(\frac{f}{z^2},-\frac{z}{x}\right)=\left[\left(\frac{f}{z^2}\right)\left(-\frac{z}{x}\right)^{2}\right]=\left[\frac{f}{x^2}\right]=1\in \frac{k(D_2)^\times}{k(D_2)^{\times 2}}
    \]
    where the last equality follows from the fact that $z=0$ and $x^3y=w^4$ on $D_2$, thus $f\mid_{D_2}=\alpha^2 xy$ and $\frac{\alpha^2 y}{x}=\alpha^2 \left(\frac{w}{x}\right)^4=\left(\alpha \frac{w^2}{x^2}\right)^2$.
    \item $\nu_3(f)=1$ and $\nu_3(x)=\nu_3(z)=0$. Hence, 
    \[
        \partial_{D_3}\left(\frac{f}{z^2},-\frac{z}{x}\right)=\left[-\frac{z}{x}\right]=\left[ \left(\frac{\alpha^3 x}{z^3} \left(w^2+\frac{xz}{\alpha}\right)\right)^2\right]=1\in \frac{k(D_3)^\times}{k(D_3)^{\times 2}}
    \]
    where the last equality follows from the following equalities on $D_3$:
    \begin{itemize}
        \item $y^3z=w^4+\frac{2}{\alpha} xz w^2+\frac{x^2z^2}{\alpha^2}=\left(w^2+\frac{xz}{\alpha}\right)^2$;
        \item $xy=-\frac{1}{\alpha^2}z^2$ implies that $y^3z=(xy)^3 \frac{z}{x}\frac{1}{x^2}=-\frac{z}{x}\left(\frac{z^2}{\alpha^2}\right)^3 \frac{1}{x^2}=\frac{-z}{x}\left(\frac{z^3}{\alpha^3 x}\right)^2$
    \end{itemize}
\end{enumerate}
Therefore, $\partial_{D_i}(\A)=0$ for all $i\in \{1,2,3\}$, hence $\A\in \br(V_\alpha)$.
    \end{proof}
    Let $\ip$ be a prime above $2$ and $\Os_{\ip}$ be the valuation ring of $k_{\ip}$. We have that $2\cdot \alpha^{-1}\in \Os_{\ip}$ if and only if $\alpha^2 \not \equiv 0 \mod 8$; we can define $\mathcal{X}_\alpha$ to be the $\Os_{\ip}$-scheme defined by equation \eqref{eq: example V_a}. If $\alpha^2\not\equiv 0 \mod 8$, then $\mathcal{X}_\alpha$ is smooth and hence $V_\alpha$ has good reduction at $\ip$; we denote by $X_\alpha$ the base change of $V_\alpha$ to $k_{\ip}$. 
    \begin{thm}\label{thm: family of examples}
        Assume that $\alpha^2\not\equiv 0 \mod 8$. Then, $\mathcal{X}_\alpha$ has good ordinary reduction if and only if $\alpha^2\equiv 1 \mod 2$. The evaluation map attached to $\A$ 
        \[
            \ev_\A\colon X_\alpha (k_{\ip})\rightarrow \Q/\Z
        \]
        is non-constant if and only if 
         \[
            \alpha^2 \not \equiv 0\mod 4.
        \]
    \end{thm}
    \begin{proof}
       Recall that we know that for K$3$ surfaces the evaluation map attached to $\A$ is non-constant if and only if $\A \notin \fil_0\br(X_\alpha)$ (cf. Section~\ref{section: the case of K3 surfaces}). 
        \begin{itemize}
            \item If $\alpha^2\equiv 1\mod 2$, then the special fibre $Y_\alpha$ is defined by the equation 
        \[
        x^3y + y^3 z+z^3 w+w^4+xyzw=0.
        \]
        From Lemma~\ref{lemma: criterium ordinary k3} we get that $Y_\alpha$ is an ordinary K$3$ surface. From Lemma~\ref{lemma: explicit global 2-form} we know that a generator (as $k(\ip)$-vector space) of $\h^0(Y_\alpha,\Omega^2_{Y_\alpha})$ is given by the global $2$-form $\omega$ that can be written (locally) as 
        \begin{align*}
            &\frac{d\left(\frac{y}{x}\right)\wedge d\left(\frac{z}{x}\right)}{\frac{z^3+xyz}{x^3}}=\frac{x^2}{z^2+xy}d\left(\frac{y}{x}\right)\wedge d\left(\frac{z}{x}\right)\frac{x}{z}=\\
            &\left(\frac{x^2}{z^2+xy}\right)\cdot d\left(\frac{z^2+xy}{x^2}\right)\wedge \left(\frac{x}{z}\right)\cdot d\left(\frac{z}{x}\right)
        \end{align*}
        where the last equality follows from 
        $d\left(\frac{z^2+xy}{x^2}\right)\wedge d\left(\frac{z}{x}\right)=d\left(\frac{y}{x}\right)\wedge d\left(\frac{z}{x}\right)$. Hence, we can write $\omega$ as $\frac{df}{f}\wedge \frac{dg}{g}$,
        with $f=\frac{z^2+xy}{x^2}$, $g=\frac{z}{x}$. Finally, we see that, by Proposition~\ref{propGradedPieces}
        \[
        [\A]=\left[\left\{\frac{z^2+\alpha^2 \cdot  xy}{z^2},-\frac{z}{x}\right\}\right]=\left[\left\{\frac{z^2+\alpha^2 \cdot  xy}{x^2},-\frac{z}{x}\right\}\right]=\rho_0(\omega,0)
        \]
        since $\frac{z^2+\alpha^2 \cdot  xy}{x^2}$ and $-\frac{z}{x}$ are lifts to characteristic $0$ of $f$ and $g$, respectively. Hence, using Proposition~\ref{prop: rsw and map rho_m} we get that $\rsw_{e',\pi}(\A)\ne (0,0)$ and $\A\notin \fil_{e'-1}\br(X_\alpha)\supseteq \fil_0\br(X_\alpha)$.
       \item If $\alpha^2\equiv 2\mod 4$, then the special fibre $Y_\alpha$ is defined by the equation 
        \begin{equation*}
            x^3y + y^3 z+z^3 w+w^4=0
        \end{equation*}
        From Lemma~\ref{lemma: criterium ordinary k3} we get that $Y_\alpha$ is a non-ordinary K$3$ surface over $k(\ip)$. From Lemma~\ref{lemma: explicit global 2-form} we know that $\h^0(Y_\alpha,\Omega^2_{Y_\alpha})$ is generated (as a $k(\ip)$-vector space) by the $2$-form $\omega$ that can be written (locally) as
        \[
            \frac{d\left(\frac{y}{x}\right)\wedge d\left(\frac{z}{x}\right)}{\frac{z^3}{x^3}}=\left(\frac{x^2}{z^2}\right)\cdot d\left(\frac{y}{x}\right)\wedge \left(\frac{x}{z}\right)\cdot  d\left(\frac{z}{x}\right)= d\left(\frac{xy}{z^2}\right)\wedge \left(\frac{x}{z}\right)\cdot d\left(\frac{z}{x}\right)
        \]
        where the last equality follows from the fact that, since we are working over a field of characteristic $2$, $\left(\frac{x}{z}\right)^2 d\left(\frac{y}{x}\right)=d\left(\frac{xy}{z^2}\right)$. Hence, in this case we can write $\omega$ as $d\left(f \cdot \frac{dg}{g}\right)$,
        with $f=\frac{xy}{z^2}$, $g=\frac{z}{x}$. Since $\alpha^2\equiv 2\mod 4$, the prime ideal $(2)$ is ramified in the field extension $\Q_2(\alpha)/\Q_2$ and $\pi=\alpha$ is a uniformiser. From Proposition~\ref{propGradedPieces} we get that 
        \[
            \left[\left\{\frac{z^2+\alpha^2 \cdot  xy}{z^2},-\frac{z}{x}\right\}\right]=\left[\left\{1+\alpha^2 \frac{xy}{z^2},-\frac{z}{x}\right\}\right]=\rho_2\left(f \cdot \frac{dg}{g}\right) 
        \]
        since $\frac{xy}{z^2}$ and $-\frac{z}{x}$ are two lifts to characteristic $0$ of $f$ and $g$, respectively.
        Hence, by Proposition~\ref{prop: rsw and map rho_m} $\rsw_{2,\pi}(\A)\ne (0,0)$. Therefore $\A\notin \fil_1\br(X_\alpha)$, and thus $\A\notin \fil_0\br(X_\alpha)\subseteq \fil_1\br(X_\alpha)$. 
        \item If $\alpha^2\equiv 0 \mod 4$, then the special fibre $Y_\alpha$ is defined by the equation 
        \[
            x^3 y+y^3z +z^3 w+w^4+xzw^2=0.
        \]
        Again, from Lemma~\ref{lemma: criterium ordinary k3} we get that $Y_\alpha$ is a non-ordinary K$3$ surface over $k(\ip)$. If $\Q_2(\alpha)/\Q_2$ is unramified then we are done by Theorem~\ref{intro_thm: good non-ordinary reduction}. If the field extension is ramified and $\pi$ is a uniformiser, then since $\alpha^2\equiv 0 \mod 4$ and $\alpha^2\not\equiv 0 \mod 8$, we have that $\alpha^2=\pi^4 \beta$ with $\beta\in \Os_{\ip}^\times$. Hence, if we look at the corresponding element in $k^2(K^h)$ via the isomorphism of equation~\eqref{eqIsomoprhismRootUnity}, we have that 
         \[
            \A\mapsto \left\{\frac{z^2+\alpha^2 \cdot  xy}{z^2},-\frac{z}{x}\right\}=\left\{1+\pi^4\cdot \frac{\beta\cdot xy}{z^2},-\frac{z}{x}\right\} \in U^{4} k^2(K^h).
        \]
         Hence, using again Proposition~\ref{prop: rsw and map rho_m}, $\A\in \fil_0\br(X_\alpha)[2]$.
         \end{itemize}
    \end{proof}
   \begin{rmk}
       The case $\alpha^2\equiv 2\mod 4$ proves that the bound appearing in Theorem~\ref{intro_thm: good non-ordinary reduction} is optimal. In fact, we are able to find examples of K$3$ surfaces over a quadratic field extension of $\Q$ such that there is a prime above $2$ whose ramification index is $e(\ip/2)=2$ and that plays a role in the Brauer--Manin obstruction to weak approximation.
       
       Finally, note that when $\alpha^2\equiv 0 \mod 4$ and $e(\ip/2)=1$, we know from Theorem~\ref{intro_thm: good non-ordinary reduction} that there is an equality $\Ev_{-1}\br(X_\alpha)=\br(X_\alpha)$. However, if $e(\ip/2)=2$, we just showed that the element $\A$ of the previous theorem lies in $\Ev_{-1}\br(X_\alpha)[2]$ and not that $\br(X_\alpha)=\Ev_{-1}\br(X_\alpha)$.
   \end{rmk}
   \newpage
   \appendix
\newpage
\bibliographystyle{plain}
\bibliography{bib.bib}
\textsc{Department of Mathematics, South Kensington Campus, Imperial College London, SW7 2AZ United Kingdom}\\
\textit{Email Address:} \textbf{m.pagano@imperial.ac.uk}
\end{document}